\documentclass[12pt,reqno]{amsart}
\usepackage{amssymb}
\usepackage{amscd}
\usepackage{hyperref}

\newcommand{\func}[1]{\operatorname{#1}}
\newtheorem{theorem}{Theorem}[section]
\newtheorem{corollary}[theorem]{Corollary}
\newtheorem{lemma}[theorem]{Lemma}
\newtheorem{proposition}[theorem]{Proposition}

\newtheorem{remark}[theorem]{Remark}

\newtheorem{example}[theorem]{Example}

\numberwithin{equation}{section}
\begin{document}
\title{The mean curvature of transverse K\"{a}hler foliations}
\author[S.~D.~Jung]{Seoung Dal Jung}
\address{Department of Mathematics\\
Jeju National University \\
Jeju 690-756 \\
Republic of Korea}
\email[S.~D.~Jung]{sdjung@jejunu.ac.kr}
\author[K.~Richardson]{Ken Richardson}
\address{Department of Mathematics \\
Texas Christian University \\
Fort Worth, Texas 76129, USA}
\email[K.~Richardson]{k.richardson@tcu.edu}

\begin{abstract}
We study properties of the mean curvature one-form and its holomorphic and
antiholomorphic cousins on a transverse K\"{a}hler foliation. If the mean
curvature of the foliation is automorphic, then there are some restrictions
on basic cohomology similar to that on K\"ahler manifolds, such as the
requirement that the odd basic Betti numbers must be even. However, the full
Hodge diamond structure does not apply to basic Dolbeault cohomology unless
the foliation is taut.
\end{abstract}

\maketitle

\renewcommand{\thefootnote}{} \footnote{%
This paper was supported by the National Research Foundation of Korea (NRF)
grant funded by the Korea government (MSIP) (NRF-2018R1A2B2002046). This
work was also supported by a grant from the Simons Foundation (Grant Number
245818 to Ken Richardson).} \renewcommand{\thefootnote}{\arabic{footnote}} %
\setcounter{footnote}{0}

\section{Introduction}

Let $\mathcal{F}$ be a foliation on a closed, smooth manifold $M$. A
Riemannian foliation is a foliation such that the normal bundle $Q=TM/T%
\mathcal{F}$ is endowed with a holonomy-invariant metric $g_{Q}$. This
metric can always be extended to a metric $g$ on $M$ that is called
bundle-like, characterized by the property that the leaves of $\mathcal{F}$
are locally equidistant. The basic forms of $(M,\mathcal{F})$ are locally
forms on the leaf space; that is, they are forms $\phi $ satisfying $%
X\lrcorner \phi =X\lrcorner d\phi =0$ for any vector $X$ tangent to the
leaves, where $X\lrcorner $ denotes the interior product with $X$. The set
of basic forms forms a differential complex and is used to define basic de
Rham cohomology groups $H_{B}^{\ast }(\mathcal{F})$. For Riemannian
foliations, these groups have finite rank, and their ranks are topological
invariants (\cite{EKN}). The basic Laplacian $\Delta _{B}$ is a version of
the Laplace operator that preserves the basic forms. Many researchers have
studied basic forms and the basic Laplacian on Riemannian foliations. It is
well-known (\cite{Al}, \cite{EKH}, \cite{KT3}, \cite{PaRi}) that on a closed
oriented manifold $M$ with a transversely oriented Riemannian foliation $%
\mathcal{F}$, $H_{B}^{r}(\mathcal{F})\cong \mathcal{H}_{B}^{r}(\mathcal{F})=%
\mathrm{ker}\Delta _{B}^{r}$.

The basic component $\kappa _{B}$ of the mean curvature one-form of the
foliation is always closed, and its cohomology class $\xi =\left[ \kappa _{B}%
\right] \in H_{B}^{1}\left( \mathcal{F}\right) $ is invariant of the choice
of bundle-like metric; this was proved by in \cite{Al}, and $\xi $ is called
the \'{A}lvarez class. Poincar\'{e} duality holds for the basic cohomology
of a Riemannian foliation $(M,\mathcal{F},g_{Q})$ if and only if the \'{A}%
lvarez class is trivial, if and only if $\left( M,\mathcal{F}\right) $ is
taut, meaning that there exists a metric for which the leaves of the
foliation are (immersed) minimal submanifolds.

In this paper, we consider foliations that admit a transverse,
holonomy-invariant complex structure, and in particular we consider
holonomy-invariant Hermitian metrics on $Q$ that may or may not be K\"{a}%
hler. The question is whether the standard K\"{a}hler manifold structures on
Dolbeault cohomology such as the Hard Lefschetz Theorem, the $dd_{c}$-Lemma,
and formality apply to the basic cohomology of transverse K\"{a}hler
foliations. The basic Dolbeault cohomologies $H_{B}^{r,s}\left( \mathcal{F}%
\right) $ and $H_{\partial _{B}\overline{\partial }_{B}}^{r,s}\left( 
\mathcal{F}\right) $ can be defined as usual using only the transverse
holomorphic structure. The Hodge diamond structure does sometimes occur for
basic Dolbeault cohomology for K\"{a}hler foliations, but it turns out that
two properties of the mean curvature are crucial:

\begin{enumerate}
\item Does the class $\eta =\left[ \partial _{B}\kappa _{B}^{0,1}\right] \in
H_{\partial _{B}\overline{\partial }_{B}}^{1,1}\left( \mathcal{F}\right) $
vanish? This class automatically vanishes if the \'{A}lvarez class $\xi $
vanishes, but it is possible for the $\xi \neq 0$ while $\eta =0$ (see
Example \ref{ExampleNewClassZeroButNotTaut}). The class is nontrivial if and
only if the $\partial _{B}\overline{\partial }_{B}$-Lemma fails to hold when
applied directly to $\partial _{B}\kappa _{B}^{0,1}$. The class $\eta $ is
an invariant of the transverse complex structure (Theorem \ref%
{newClassDefinedTheorem}).

\item Is the mean curvature $H=\kappa _{B}^{\#}$ automorphic (that is, does
its flow preserve the transverse complex structure)?
\end{enumerate}

The condition (2) is equivalent to $(\kappa _{B}^{0,1})^{\#}=H^{1,0}$ being
a transverse holomorphic vector field --- that locally it has the form 
\begin{equation*}
H^{1,0}=\sum_j H^{1,0}_j(z)\partial_{z_j}
\end{equation*}
in the transverse holomorphic coordinates, with each $H^{1,0}_j(z)$ being a
transverse holomorphic function. Condition (1) is similar; only in that case
each $H^{1,0}_j(z)$ is required to be a transverse antiholomorphic function.

\vspace{1pt}This paper is organized as follows. In Section \ref%
{meanCurvRiemFolsSection}, we review the known properties of the mean
curvature and basic Laplacian for Riemannian foliations. In Section \ref%
{transHermSection}, we investigate transverse Hermitian structures on
foliations with bundle-like metrics. In Proposition \ref%
{holomorphicKappaClosedProp} and Theorem \ref{holomorphicAlvarezClass}, we
show that the holomorphic and antiholomorphic basic components $\kappa
_{B}^{1,0}$ and $\kappa _{B}^{0,1}$ of mean curvature are $\partial _{B}$%
-closed, $\overline{\partial }_{B}$-closed and represent basic Dolbeault
cohomology classes in $H_{\partial _{B}}^{1,0}\left( \mathcal{F}\right) $
and $H_{\overline{\partial }_{B}}^{0,1}\left( \mathcal{F}\right) $,
respectively, that are invariant under the choices of bundle-like metric and
transverse metric that is compatible with a given transverse holomorphic
structure. In fact, in Proposition \ref{DomGen}, we show that the metrics
can be chosen so that $\kappa $, $\kappa ^{1,0}$, and $\kappa ^{0,1}$ are
basic forms. When the foliation is transversally K\"{a}hler, the metrics can
be chosen so that $\kappa $, $\kappa ^{1,0}$, and $\kappa ^{0,1}$ are
simultaneously basic, $\Delta _{B}$-harmonic, $\square _{B}$-harmonic, and $%
\overline{\square }_{B}$-harmonic, respectively (see Proposition \ref%
{harmonicMeanCurvKahlerThm}; the $\square _{B}$ and $\overline{\square }_{B}$
are the $\partial _{B}$ and $\overline{\partial }_{B}$ Laplacians,
respectively).

In Section \ref{newClassSection}, we show that for transverse Hermitian
foliations, the form $\partial _{B}\kappa _{B}^{0,1}$ is $\partial _{B}%
\overline{\partial }_{B}$-closed and generates a class $\eta =\left[
\partial _{B}\kappa _{B}^{0,1}\right] \in H_{\partial _{B}\overline{\partial 
}_{B}}^{1,1}\left( \mathcal{F}\right) $ that is invariant of the choices of
bundle-like metric and compatible transverse metric (Theorem \ref%
{newClassDefinedTheorem}). If the foliation is not taut and is transversally
Hermitian, Proposition \ref{nontautAndCLassZeroImpliesH1gt1Prop} implies
that if $\partial _{B}\kappa _{B}^{0,1}=0$, then $\left[ \kappa _{B}\right] $
and $\left[ J\kappa _{B}\right] $ are linearly independent cohomology
classes in $H_{B}^{1}\left( \mathcal{F}\right) $, so that $\dim
H_{B}^{1}\left( \mathcal{F}\right) \geq 2$ in this case. In Proposition \ref%
{FormulasForBoxB}, we derive formulas for $\square _{B}$ and $\overline{%
\square }_{B}$ that are valid for all transverse Hermitian foliations; for
example,%
\begin{equation*}
\overline{\square }_{B}=\Delta _{\bar{\partial}}^{Q}+\bar{\partial}_{B}\circ
H^{0,1}\lrcorner +H^{0,1}\lrcorner \circ \bar{\partial}_{B},
\end{equation*}%
where $H^{0,1}=\left( \kappa _{B}^{1,0}\right) ^{\#}$ and where $\Delta _{%
\bar{\partial}}^{Q}$ is the Dolbeault Laplacian on the local quotients of
foliation charts. In Corollary \ref{FormulasForBoxB}, we show for example
that if the foliation is transversely K\"{a}hler, then $\partial _{B}\kappa
_{B}^{0,1}=0$ if and only if $\overline{\square }_{B}=\Delta _{\bar{\partial}%
}^{Q}+\nabla _{H^{0,1}}$.

In Section \ref{LefschetzDecompSection}, we investigate the properties of
the operator $L$, exterior product with the transverse K\"{a}hler form. As a
result, we show in Lemma \ref{LaplacianKahlerFormLemma} that for transverse K%
\"{a}hler foliations,%
\begin{equation*}
\Delta _{B}=\square _{B}+\overline{\square }_{B}+\func{Re}\left( \partial
_{B}H^{0,1}\lrcorner +H^{0,1}\lrcorner \,\partial _{B}\right) .
\end{equation*}%
As a consequence (Corollary \ref{weakInequalityCorollary}),%
\begin{equation*}
\dim \left( \mathcal{H}_{B}^{j}\left( \mathcal{F}\right) \cap \Omega
_{B}^{r,s}\left( \mathcal{F}\right) \right) \leq \dim H_{B}^{r,s}\left( 
\mathcal{F}\right)
\end{equation*}%
for all transverse K\"{a}hler foliations. Then it turns out that the Hard
Lefschetz Theorem for basic cohomology holds if and only if the class $\eta
=0$ (Theorem \ref{HardLefschetzTheorem} and Corollary \ref%
{classTrivialImpliesAlvClassTrivialCor}). Further, for transverse K\"{a}hler
foliations, $\eta =0$ if and only if $\xi =0$, which is false in general,
and this condition in turn implies that the metric can be chosen so that $%
\kappa =\kappa _{B}=0$.

\vspace{1pt}In Section \ref{automorphMeanCurvSection}, we investigate the
condition that the mean curvature $\kappa _{B}$ is automorphic, meaning its
flow preserves the transverse holomorphic structure. The basic Laplacian
satisfies $\Delta _{B}=\square _{B}+\overline{\square }_{B}$ if and only if
it preserves the $\left( r,s\right) $-type of form if and only if the mean
curvature is automorphic (Theorem \ref{automorphicMeanCurvTheorem} and
Corollary \ref{automorphicIFFpreserveTypeOfFormCor}).

In Section \ref{ddcSection}, we show that the $dd_{c}$-Lemma of K\"{a}hler
manifolds works only for taut K\"{a}hler foliations (Lemma \ref{ddcLemma}).
In Section \ref{HodgeDiamondSection}, Theorem \ref{HodgeDiamondTheorem}
shows on transverse K\"{a}hler foliations that if the mean curvature is
automorphic, then symmetry of a version of the Hodge diamond follows, and
then we also have%
\begin{equation*}
\dim H_{B}^{j}\left( \mathcal{F}\right) =\sum\limits_{r+s=j;~r,s\geq 0}\dim 
\mathcal{H}_{\Delta _{B}}^{r,s}\left( \mathcal{F}\right) .
\end{equation*}%
However, the full power of the Hodge diamond with restrictions to basic
Dolbeault cohomology follows from the Hard Lefschetz theorem, which applies
only if the foliation is taut.

In Section \ref{examplesSection}, we provide examples of nontaut K\"{a}hler
foliations and calculate their cohomologies. Also in this section, we show
that for a nontaut foliation, it is possible for one transverse Hermitian
structure to be K\"{a}hler with $\eta \neq 0$ and mean curvature not
automorphic and for another transverse Hermitian structure to be nonK\"{a}%
hler with $\eta =0$ and mean curvature automorphic. These examples manifest
another interesting property of nontaut transverse K\"{a}hler foliations;
the K\"{a}hler form $\omega $ always yields a transverse volume form $\omega
^{n}$ that is exact, and the K\"{a}hler form itself may be exact.

Foliations that admit a transverse K\"{a}hler structure have been studied by
many researchers, but primarily in the case when the foliation is taut ($%
\kappa =0$ for some metric). For example, Sasaki manifolds are not K\"{a}%
hler but admit transverse K\"{a}hler structures on the characteristic
foliation, which is homologically oriented. Since the mean curvature
vanishes, many K\"{a}hler manifold facts apply to the basic Dolbeault
cohomology (see \cite[Section 2]{BGN}, \cite[Proposition 7.2.3]{BG}, \cite%
{Wo}). The authors in \cite{CMNY} prove the hard Lefschetz theorem for
compact Sasaki manifolds, which again is a simple case of the results of
this paper with $\kappa =0$. The cosymplectic manifold case is treated in 
\cite{ChdLM}. A. El Kacimi proved in \cite[Section 3.4]{EK} that the
standard facts about K\"{a}hler manifolds and their cohomology carry over to
basic cohomology in the homologically orientable (taut) case. Also, L. A.
Cordero and R. A. Wolak \cite{CW} studied basic cohomology on taut
transverse K\"{a}hler foliations by using the differential operator $\Delta
_{T}$, which is different from $\Delta _{B}$ ($\mathcal{F}$ is minimal if
and only if $\Delta _{T}=\Delta _{B}$). We note other recent work on
transverse K\"{a}hler foliations in \cite{JP2}, \cite{JL}, \cite{HV}, \cite%
{KLW}, \cite{JR2}.

\section{Properties of the mean curvature for Riemannian foliations\label%
{meanCurvRiemFolsSection}}

\vspace{0in}Let $(M,g_{Q},\mathcal{F})$ be a $(p+q)$-dimensional Riemannian
foliation of codimension $q$ with compact leaf closures. Here, $g_{Q}$ is a
holonomy invariant metric on the normal bundle $Q=TM/T\mathcal{F}$, meaning
that $\mathcal{L}_{X}g_{Q}=0$ for all $X\in T\mathcal{F}$, where $\mathcal{L}%
_{X}$ denotes the Lie derivative with respect to $X$. Next, let $g_{M}$ be a
bundle-like metric on $M$ adapted to $g_{Q}$. This means that if $T\mathcal{F%
}^{\perp }$ is the $g_{M}$-orthogonal complement to $T\mathcal{F}$ in $TM$
and $\sigma :Q\rightarrow T\mathcal{F}^{\perp }$ is the canonical bundle
isomorphism, then $g_{Q}=\sigma ^{\ast }\left( \left. g_{M}\right\vert _{T%
\mathcal{F}^{\perp }}\right) $. We do not assume that $M$ is compact, but we
assume it is complete with finite volume.

In this section, we review some known results for this Riemannian foliation
setting. Let $\nabla$ be the transverse Levi-Civita connection on the normal
bundle $Q$, which is torsion-free and metric with respect to $g_Q$ \cite{TO2}%
. %Let $\pi :TM\rightarrow Q$ be the bundle projection. The transversal
%Levi-Civita connection $\nabla $ on $Q\rightarrow M$ is given for all $s\in
%\Gamma (Q)$ by 
%\begin{equation*}
%\nabla _{X}s=\left\{ 
%\begin{split}
%& \pi ([X,Y_{s}])\qquad \forall X\in \Gamma (T\mathcal{F}) \\
%& \pi (\nabla _{X}^{M}Y_{s})\qquad \forall X\in \Gamma (T\mathcal{F}^{\perp
%}),
%\end{split}%
%\right.
%\end{equation*}%
%where $\sigma \left( s\right) =Y_{s}\in \Gamma (T\mathcal{F}^{\perp })$. 
Let $R^Q$ and $\mathrm{Ric}^{Q}$ be the curvature tensor and the transversal
Ricci operator of $\mathcal{F}$ with respect to $\nabla$, respectively. The
mean curvature vector $\tau $ of $\mathcal{F}$ is given by 
\begin{equation}
\tau= \sum_{i=1}^{p}\pi (\nabla _{f_{i}}^{M}f_{i}),
\end{equation}%
where $\{f_{i}\}_{i=1,\cdots ,p}$ is a local orthonormal basis of $T\mathcal{%
F}$ and $\pi:TM\to Q$ is natural projection. Then the \textit{mean curvature
form} $\kappa$ is defined by 
\begin{equation}
\kappa(X) = g_Q (\tau,\pi(X))
\end{equation}
for any tangent vector $X\in \Gamma (TM)$. An $r$-form $\phi \ $is basic if
and only if $X\lrcorner \,\phi =0$ and $\mathcal{L}_{X}\phi =0$ for any $%
X\in \Gamma \left( T\mathcal{F}\right) $, where $X\lrcorner $ denotes the
interior product. Let $\Omega _{B}^{r}(\mathcal{F})$ be the space of all 
\textit{basic }$r $ -\textit{forms}. %In the sequel we
%will also use the notation $\Omega _{B,0}^{r}(\mathcal{F})$ to denote
%compactly supported basic forms. 
The foliation $\mathcal{F}$ is said to be \textit{minimal} if $\kappa =0$.
We note that Rummler's formula (from \cite{Rum}) for the mean curvature is 
\begin{equation}
d\chi _{\mathcal{F}}=-\kappa \wedge \chi _{\mathcal{F}}+\varphi _{0}\text{
with }\chi _{\mathcal{F}}\wedge \ast \varphi _{0}=0,  \label{Rummler}
\end{equation}%
where $\chi _{\mathcal{F}}:=f_{1}^{\flat }\wedge ...\wedge f_{p}^{\flat }$
is the \emph{characteristic form}, the leafwise volume form, and $\ast $ is
the Hodge star operator associated to $g_{M}$; we assumed $M$ is oriented to
make the property of $\varphi _{0}$ easier to state.

The exterior derivative $d$ maps $\Omega _{B}^{r}(\mathcal{F})$ to $\Omega
_{B}^{r+1}(\mathcal{F})$, and the resulting cohomology groups are called the
basic cohomology groups: for $r\geq 0$,%
\begin{equation*}
H_{B}^{r}\left( \mathcal{F}\right) =\frac{\ker \left( \left. d\right\vert
_{\Omega _{B}^{r}(\mathcal{F})}\right) }{\func{Im}\left( \left. d\right\vert
_{\Omega _{B}^{r-1}(\mathcal{F})}\right) }.
\end{equation*}%
These groups are smooth invariants of the foliation and do not depend on the
bundle-like metric and also do not even depend on the smooth foliation
structure (see \cite{EKN}).

The metric $g_{M}$ induces a natural metric on $\Lambda ^{\ast }T^{\ast }M$
and $L^{2}$ metric on $L^{2}\Omega ^{\ast }\left( M\right) $. Let $%
L^{2}\Omega _{B}^{\ast }(\mathcal{F})$ denote the closure of $\Omega
_{B,0}^{\ast }(\mathcal{F})$, the space of compactly supported basic forms,
in $L^{2}\Omega ^{\ast }\left( M\right) $.

\begin{proposition}
(Proved in \cite{PaRi} for the closed manifold case) Let \vspace{0in}$%
(M,g_{Q},\mathcal{F})$ be a Riemannian foliation with compact leaf closures
and bundle-like metric. The orthogonal projection $P:L^{2}\Omega ^{\ast
}\left( M\right) \rightarrow L^{2}\Omega _{B}^{\ast }\left( \mathcal{F}%
\right) $ maps smooth forms to smooth basic forms. For all $\alpha \in
L^{2}\Omega ^{\ast }\left( M\right) $, $P\left( \alpha \right) \left(
x\right) $ is computed by an integral over the leaf closure containing $x$
and only depends on the values of $\alpha $ on that leaf closure.
\end{proposition}

\begin{proof}
The proof in \cite{PaRi} applies in this slightly more general case, where
it is not assumed that $M$ is compact.
\end{proof}

Now we recall the transversal star operator $\bar{\ast}:\Omega _{B}^{r}(%
\mathcal{F})\rightarrow \Omega _{B}^{q-r}(\mathcal{F})$ given by 
\begin{equation*}
\bar{\ast}\phi =(-1)^{p(q-r)}\ast (\phi \wedge \chi _{\mathcal{F}}),
\end{equation*}%
where $\ast $ is the Hodge star operator associated to $g_{M}$; this is
actually well-defined as long as $(M,\mathcal{F})$ is transversely oriented.
Trivially, $\bar{\ast}^{2}\phi =(-1)^{r(q-r)}\phi $ for any $\phi \in \Omega
_{B}^{r}(\mathcal{F})$. Let $\nu $ be the transversal volume form; that is, $%
\ast \nu =\chi _{\mathcal{F}}$ as long as $M$ is oriented. Then the
pointwise inner product $\langle \cdot ,\cdot \rangle $ on $\Lambda
^{r}Q^{\ast }$ is defined by $\langle \phi ,\psi \rangle \nu =\phi \wedge 
\bar{\ast}\psi $. The global inner product on $L^{2}\Omega _{B}^{r}(\mathcal{%
F})$ is 
\begin{equation*}
\ll \phi ,\psi \gg \, =\int_{M}\langle \phi ,\psi \rangle \mu
_{M}=\int_{M}\phi \wedge \bar{\ast}\psi \wedge \chi _{\mathcal{F}},
\end{equation*}%
where $\mu _{M}=\nu \wedge \chi _{\mathcal{F}}$ is the volume form with
respect to $g_{M}$.

In what follows, let $\kappa _{B}=P\kappa $. Also, let 
\begin{equation}
d_{B}=d|_{\Omega _{B}^{\ast }(\mathcal{F})},~~d_{T}=d_{B}-\epsilon (\kappa
_{B}),  \label{dTdBFormulas}
\end{equation}%
where $\epsilon (\alpha )\psi =\alpha \wedge \psi $ for any $\alpha \in
\Omega _{B}^{1}(\mathcal{F})$.  
%The operator $\epsilon $ acts pointwise and of
%course is defined on $Q^{\ast }\cong \left( \left( T\mathcal{F}\right)
%^{\bot }\right) ^{\ast }.$ 
The interior product $v\lrcorner $ of $v\in Q\cong T\mathcal{F}^{\bot }$ on
differential forms satisfies%
\begin{equation*}
v\lrcorner =\epsilon \left( v^{\flat }\right) ^{\ast },
\end{equation*}%
where $\ast $ denotes the pointwise adjoint.

\begin{proposition}
\label{basicAdjointFormulasProp}(In \cite{Al}, \cite{PaRi} for the compact
case; \cite[Prop. 2.1]{JR2}) Let \vspace{0in}$(M,g_{Q},\mathcal{F})$ be a
Riemannian foliation with compact leaf closures and bundle-like metric. The
formal adjoint operators $\delta _{B}$ and $\delta _{T}$ of $d_{B}$ and $%
d_{T}$ with respect to $\ll \cdot ,\cdot \gg $ on basic forms are given by 
\begin{equation*}
\delta _{B}\phi =(-1)^{q(r+1)+1}\bar{\ast}d_{T}\bar{\ast}\phi =\left( \delta
_{T}+\kappa _{B}^{\sharp }\lrcorner \,\right) \phi ,\quad \delta _{T}\phi
=(-1)^{q(r+1)+1}\bar{\ast}d_{B}\bar{\ast}\phi ,
\end{equation*}%
on basic $r$-forms $\phi $.
\end{proposition}

\begin{lemma}
\label{divergenceLemma}The transversal divergence satisfies%
\begin{equation*}
\delta _{T}=-\sum_{a=1}^{q}\left( E_{a}\lrcorner \right) \nabla _{E_{a}},
\end{equation*}%
where the sum is over a local orthonormal frame $\left\{ E_{a}\right\} $ of $%
Q$.
\end{lemma}

\begin{proof}
It follows from the fact that $\delta _{T}$ is locally the pullback of the
ordinary divergence on the local quotients of foliation charts.
\end{proof}

\begin{proposition}
\label{AlvarezClassThm}(Proved in \cite{Al} for the closed manifold case)
Let \vspace{0in}$(M,g_{Q},\mathcal{F})$ be a Riemannian foliation with
compact leaf closures and bundle-like metric. The form $d\kappa _{B}=0$, and 
$\kappa _{B}$ determines a class in $H_{B}^{1}\left( \mathcal{F}\right) $
that is independent of the choice of $g_{M}$ or of $g_{Q}$.
\end{proposition}

\begin{proof}
\vspace{0in}The proof in \cite{Al} is primarily a calculation confined to a
neighborhood of a leaf closure, so that it applies in this slightly more
general case. For the sake of exposition, we show the proof that $\kappa
_{B} $ is closed: We have%
\begin{equation*}
\delta _{B}=\delta _{T}+\kappa _{B}^{\sharp }\lrcorner \,
\end{equation*}%
where $\delta _{T}$ is the divergence on the local quotient manifolds in the
foliation charts. In particular, $\delta _{T}$ only depends on $g_{Q}$. Thus,%
\begin{equation*}
\delta _{T}^{2}=0,
\end{equation*}%
and also 
\begin{equation*}
d_{B}^{2}=0.
\end{equation*}%
Taking adjoints with respect to basic forms, from the three equations above
we have%
\begin{equation*}
d_{T}= \delta _{T} ^{\ast }=d_{B}-\epsilon \left( \kappa _{B}\right)
\end{equation*}%
\begin{equation*}
d_{T}^{2}=0,~\left( \delta _{B}\right) ^{2}=0.
\end{equation*}%
Then 
\begin{eqnarray*}
d_{T}\left( 1\right) &=&\left( d_{B}-\epsilon (\kappa _{B})\right) \left(
1\right) \\
&=&-\kappa _{B}
\end{eqnarray*}%
and%
\begin{eqnarray*}
d_{B}\kappa _{B} &=&\left( d_{T}+\epsilon (\kappa _{B})\right) \kappa _{B} \\
&=&d_{T}\kappa _{B} \\
&=&d_{T}\left( -d_{T}1\right) =0.
\end{eqnarray*}
\end{proof}

\begin{proposition}
\label{tensenessTheorem}(\cite{Dom})Let \vspace{0in}$(M,g_{Q},\mathcal{F})$
be a Riemannian foliation on a closed manifold. Then there exists a
bundle-like metric compatible with $g_{Q}$ such that $\kappa $ is a basic
form; that is, $\kappa =\kappa _{B}$.
\end{proposition}

\begin{proposition}
\label{harmonicMeanCurvThm}(\cite{MO} and \cite{MA}) Let \vspace{0in}$%
(M,g_{Q},\mathcal{F})$ be a Riemannian foliation on a closed manifold. Then
there exists a bundle-like metric compatible with $g_{Q}$ such that $\kappa $
is basic harmonic; that is, $\kappa =\kappa _{B}$ and $\delta _{B}\kappa =0$.
\end{proposition}

\begin{corollary}
\label{harmonicMeanCurvCorollary}Let \vspace{0in}$(M,g_{Q},\mathcal{F})$ be
a Riemannian foliation on a closed manifold, and let $\alpha $ be any
element of the class $\left[ \kappa _{B}\right] \in H_{B}^{1}\left( \mathcal{%
F}\right) .$ Then there exists a bundle-like metric compatible with $g_{Q}$
such that $\kappa =\alpha $. The representative $\alpha $ corresponding to a
bundle-like metric such that $\alpha =\kappa $ is basic harmonic is uniquely
determined. For that metric, $\kappa $ is the element of $\left[ \kappa _{B}%
\right] $ of minimum $L^{2}$-norm.
\end{corollary}

\begin{proof}
Given any bundle-like metric with basic mean curvature $\kappa $ as in
Proposition \ref{tensenessTheorem}, any element of $\left[ \kappa \right] =%
\left[ \kappa _{B}\right] $ is of the form $\kappa +df$ for some basic
function $f$. If $p=\dim \mathcal{F}$, multiplying the leafwise metric by $%
e^{-(2/p)f}$ yields a new characteristic form $\chi _{\mathcal{F}}^{\prime
}=e^{-f}\chi _{\mathcal{F}}$ so that the new mean curvature form from (\ref%
{Rummler}) satisfies 
\begin{eqnarray*}
-\kappa ^{\prime }\wedge \chi _{\mathcal{F}}^{\prime }+\varphi _{0}^{\prime
} &=&d\left( \chi _{\mathcal{F}}^{\prime }\right) \\
&=&-df\wedge \chi _{\mathcal{F}}^{\prime }+e^{-f}d\chi _{\mathcal{F}} \\
&=&-\left( \kappa +df\right) \wedge \chi _{\mathcal{F}}^{\prime }+\varphi
_{0}^{\prime }.
\end{eqnarray*}%
The second part comes from the proof in \cite{MO}, where the volume form $%
\nu \wedge \chi _{\mathcal{F}}^{\prime }$ is uniquely determined (up to
rescaling, which does not change $\kappa $). The third part comes from the
fact that if $\delta _{B}\kappa =d\kappa =0$,%
\begin{eqnarray*}
\ll \kappa +df,\kappa +df\gg &=&\ll \kappa ,\kappa \gg +2\ll df,\kappa \gg
+\ll df,df\gg \\
&=&\ll \kappa ,\kappa \gg +2\ll f,\delta _{B}\kappa \gg +\ll df,df\gg \\
&=&\ll \kappa ,\kappa \gg +\ll df,df\gg .
\end{eqnarray*}
\end{proof}

The \textit{basic Laplacian} $\Delta _{B}$ is the operator on basic forms
defined as 
\begin{equation*}
\Delta _{B}=\delta _{B}d_{B}+d_{B}\delta _{B}.
\end{equation*}%
We define the operator $\Delta _{T}$ on basic forms as the corresponding
Laplacian on the local quotient manifolds. Specifically,%
\begin{equation*}
\Delta _{T}=\delta _{T}d_{B}+d_{B}\delta _{T}.
\end{equation*}%
The operator $\Delta _{T}$ is not essentially self-adjoint on the space of
basic forms, but the operator $\Delta _{B}$ is.

\begin{lemma}
The basic Laplacian is the restriction of the operator%
\begin{equation*}
\Delta _{B}=\Delta _{T}+\mathcal{L}_{\kappa _{B}^{\#}}.
\end{equation*}
\end{lemma}

\begin{proof}
From Proposition \ref{basicAdjointFormulasProp}, 
\begin{eqnarray*}
\Delta _{B} &=&\left( \delta _{T}+\kappa _{B}^{\#}\lrcorner \right)
d_{B}+d_{B}\left( \delta _{T}+\kappa _{B}^{\#}\lrcorner \right) \\
&=&\Delta _{T}+\left( \kappa _{B}^{\#}\lrcorner \right) d_{B}+d_{B}\left(
\kappa _{B}^{\#}\lrcorner \right) .
\end{eqnarray*}%
The result follows from Cartan's formula for the Lie derivative.
\end{proof}

\section{Properties of the mean curvature for transverse Hermitian
foliations \label{transHermSection}}

We now suppose that $\left( M,\mathcal{F}\right) $ is a foliation of
codimension $2n$ and is endowed with a holonomy-invariant transverse complex
structure $J:Q\rightarrow Q$ and a holonomy-invariant Hermitian metric on
the complexified normal bundle; we call such a foliation a \emph{transverse
Hermitian foliation}. So in particular the foliation is Riemannian. When it
is convenient, we will also refer to the bundle map $J^{\prime
}:TM\rightarrow TM$ defined by $J^{\prime }(v)=J(\pi (v))$ and abuse
notation by denoting $J=J^{\prime }$. In what follows, we use notation
similar to \cite{JR2}.

For $Q^{C}=Q\otimes \mathbb{C}$, we let 
\begin{equation*}
Q^{1,0}=\{Z\in Q^{C}|JZ=iZ\},\quad Q^{0,1}=\{Z\in Q^{C}|JZ=-iZ\}.
\end{equation*}%
Elements of $Q^{1,0}$ and $Q^{0,1}$ are called \textit{complex normal vector
fields of type} $(1,0)$ and $\left( 0,1\right) $, respectively. We have $%
Q^{C}=Q^{1,0}\oplus Q^{0,1}$ and 
\begin{equation*}
Q^{1,0}=\{X-iJX|\ X\in Q\},\quad Q^{0,1}=\{X+iJX|\ X\in Q\}.
\end{equation*}%
Let $Q_{C}^{\ast }$ be the real dual of $Q^{C}$; at each $x\in M$, $\left(
Q_{C}^{\ast }\right) _{x}$ is set of $\mathbb{C}$-linear maps from $Q_{x}^{C}
$ to $\mathbb{C}$. Letting $\Lambda _{C}Q^{\ast }$ denote $\Lambda
Q_{C}^{\ast }$, we decompose $\Lambda _{C}^{1}Q^{\ast }=Q_{1,0}\oplus Q_{0,1}
$, where the sub-bundles $Q_{1,0}$ and $Q_{0,1}$ are given by 
\begin{align*}
Q_{1,0}& =\{\xi \in \Lambda _{C}^{1}Q^{\ast }|\ \xi (Z)=0,\ \forall Z\in
Q^{0,1}\}, \\
Q_{0,1}& =\{\xi \in \Lambda _{C}^{1}Q^{\ast }|\ \xi (Z)=0,\ \forall Z\in
Q^{1,0}\}.
\end{align*}%
Also 
\begin{equation*}
Q_{1,0}=\{\theta +iJ\theta |\ \theta \in Q^{\ast }\},\quad Q_{0,1}=\{\theta
-iJ\theta |\ \theta \in Q^{\ast }\},
\end{equation*}%
where $(J\theta )(X):=-\theta (JX)$ for any $X\in Q$ and is extended
linearly. Let $\Omega _{B}^{r,s}(\mathcal{F})$ be the set of the basic forms
of type $(r,s)$, the smooth sections of $\Lambda _{C}^{r,s}Q^{\ast }$, which
is the subspace of $\Lambda _{C}Q^{\ast }$ spanned by $\xi \wedge \eta $,
where $\xi \in \Lambda ^{r}Q_{1,0}$ and $\eta \in \Lambda ^{s}Q_{0,1}$. We
choose $\{E_{a},JE_{a}\}_{a=1,\cdots ,n}$ so that it is a local orthonormal
basic frame; we call it a\emph{\ }$J$\emph{-basic frame}) of $Q$. Let $%
\{\theta ^{a},J\theta ^{a}\}_{a=1,\cdots ,n}$ be the local dual frame of $%
Q^{\ast }$. We set $V_{a}={\frac{1}{\sqrt{2}}}(E_{a}-iJE_{a})$ and $\omega
^{a}={\frac{1}{\sqrt{2}}}(\theta ^{a}+iJ\theta ^{a})$, so that 
\begin{equation*}
\omega ^{a}(V_{b})=\bar{\omega}^{a}(\bar{V}_{b})=\delta _{ab},\ \omega ^{a}(%
\bar{V}_{b})=\bar{\omega}^{a}(V_{b})=0.
\end{equation*}%
The frame $\{V_{a}\}$ is a local orthonormal basic frame field of $Q^{1,0}$,
a \emph{normal frame field of type }$(1,0)$, and $\{\omega ^{a}\}$ is the
corresponding dual coframe field.

The following 2-form $\omega $ is nondegenerate. Letting $\theta
^{n+a}=J\theta ^{a}$ for $a=1,...,n$,%
\begin{eqnarray*}
\omega &=&-\frac{1}{2}\sum_{a=1}^{2n}\theta ^{a}\wedge J\theta ^{a} \\
&=&-\frac{1}{2}\left( \sum_{a=1}^{n}\theta ^{a}\wedge J\theta
^{a}+\sum_{a=1}^{n}J\theta ^{a}\wedge J^{2}\theta ^{a}\right) \\
&=&-\sum_{a=1}^{n}\theta ^{a}\wedge J\theta ^{a}.
\end{eqnarray*}%
In the event that $\omega $ is closed, this is the K\"{a}hler form, and the
foliation is \emph{transversely K\"{a}hler}.

\vspace{0in}We define $\left. \partial _{B}\right\vert _{\Omega _{B}^{r,s}(%
\mathcal{F})}=\Pi ^{r+1,s}\left. d_{B}\right\vert _{\Omega _{B}^{r,s}(%
\mathcal{F})}$, where $\Pi ^{r,s}:\Omega _{B}^{r+s}\left( \mathcal{F}\right)
\rightarrow \Omega _{B}^{r,s}\left( \mathcal{F}\right) $ is the projection,
and similarly $\left. \overline{\partial }_{B}\right\vert _{\Omega
_{B}^{r,s}(\mathcal{F})}=\Pi ^{r,s+1}\left. d_{B}\right\vert _{\Omega
_{B}^{r,s}(\mathcal{F})}$. Similarly, we define $\partial _{T}$ and $\bar{%
\partial}_{T}$, using $d_{T}$ from (\ref{dTdBFormulas}). We now write $%
\kappa _{B}=\kappa _{B}^{1,0}+\kappa _{B}^{0,1}$, with 
\begin{equation*}
\kappa _{B}^{1,0}=\frac{1}{2}(\kappa _{B}+iJ\kappa _{B})\in \Omega
_{B}^{1,0}(\mathcal{F}),~\kappa _{B}^{0,1}=\overline{\kappa _{B}^{1,0}}\in
\Omega _{B}^{0,1}(\mathcal{F}).
\end{equation*}%
Let $H=\kappa _{B}^{\#}$ be the basic mean curvature vector field, and let%
\begin{eqnarray}
H^{1,0} &:&=\left( \kappa _{B}^{1,0}\right) ^{\ast }=\left( \overline{\kappa
_{B}^{1,0}}\right) ^{\#}=\frac{1}{2}(\kappa _{B}^{\#}-iJ\kappa _{B}^{\#})\in
\Gamma \left( Q^{1,0}\right) ,  \label{HKappaFormulas} \\
H^{0,1} &:&=\overline{H^{1,0}}\in \Gamma \left( Q^{1,0}\right) .  \notag
\end{eqnarray}

In what follows, we extend the definitions of exterior product and interior
product linearly to complex vectors and differential forms. Observe that $%
V\lrcorner $ is by definition the adjoint of $\epsilon \left( V^{\flat
}\right) $ on real vector fields, but on complex vector fields the following
holds. If $v,w$ are real tangent vectors,%
\begin{eqnarray*}
\left( v+iw\right) \lrcorner  &=&\left( v\lrcorner \right) +i\left(
w\lrcorner \right) =\left( \epsilon \left( v^{\flat }\right) -i\epsilon
\left( w^{\flat }\right) \right) ^{\ast } \\
&=&\left( \epsilon \left( v^{\flat }-iw^{\flat }\right) \right) ^{\ast },
\end{eqnarray*}%
so that for complex vectors $V$,%
\begin{equation*}
\left( V\lrcorner \right) ^{\ast }=\epsilon \left( \overline{V}^{\flat
}\right) ,~\left( \epsilon \left( V^{\flat }\right) \right) ^{\ast }=%
\overline{V}\lrcorner .
\end{equation*}

Now, let $\langle \cdot ,\cdot \rangle $ be a Hermitian inner product on $%
\Lambda _{C}^{r,s}(\mathcal{F})$ induced by the transverse Hermitian
structure, and let $\bar{\ast}:\Lambda _{C}^{r,s}(\mathcal{F})\rightarrow
\Lambda _{C}^{n-s,n-r}(\mathcal{F})$ be the star operator defined by 
\begin{equation*}
\phi \wedge \bar{\ast}\bar{\psi}=\langle \phi ,\psi \rangle \nu ,
\end{equation*}%
where $\nu $ is the transverse volume form. Then for any $\psi \in \Lambda
_{C}^{r,s}(\mathcal{F})$, 
\begin{equation*}
\overline{\bar{\ast}\psi }=\bar{\ast}\bar{\psi},\quad \bar{\ast}^{2}\psi
=(-1)^{r+s}\psi .
\end{equation*}%
Then for complex vectors $V$ it follows that%
\begin{equation*}
\bar{\ast}\epsilon \left( V^{\flat }\right) \bar{\ast}=\overline{V}\lrcorner
,~\bar{\ast}\left( V\lrcorner \right) \bar{\ast}=-\epsilon \left( \overline{V%
}^{\flat }\right) .
\end{equation*}

Now, by applying the projections to $\Pi ^{\ast ,\ast }$ to (\ref%
{dTdBFormulas}), we have 
\begin{equation*}
\partial _{T}=\partial _{B}-\epsilon (\kappa _{B}^{1,0}),\quad \bar{\partial}%
_{T}=\bar{\partial}_{B}-\epsilon (\kappa _{B}^{0,1}).
\end{equation*}%
Then, since $\left( M,\mathcal{F},J\right) $ is transversely holomorphic,%
\begin{equation*}
d_{B}=\partial _{B}+\bar{\partial}_{B},~d_{T}=\partial _{T}+\bar{\partial}%
_{T}.
\end{equation*}%
Taking $L^{2}$ adjoints with respect to basic forms of the formulas above,
we have%
\begin{equation*}
\delta _{B}=\partial _{B}^{\ast }+\bar{\partial}_{B}^{\ast },~\delta
_{T}=\partial _{T}^{\ast }+\bar{\partial}_{T}^{\ast }.
\end{equation*}

\begin{proposition}
\label{delBstarFormulaProp}\cite[Prop. 3.6]{JR2} Let \vspace{0in}$(M,g_{Q},%
\mathcal{F})$ be a transverse Hermitian foliation with compact leaf closures
and bundle-like metric. The formal adjoint operators $\partial _{B}^{\ast }$
and $\partial _{T}^{\ast }$ of $\partial _{B}$ and $\partial _{T}$ with
respect to $\ll \cdot ,\cdot \gg $ on basic forms are given by 
\begin{equation*}
\partial _{B}^{\ast }\phi =-\bar{\ast}\partial _{T}\bar{\ast}\phi =\left(
\partial _{T}^{\ast }+H^{1,0}\lrcorner \,\right) \phi ,\quad \partial
_{T}^{\ast }\phi =-\bar{\ast}\partial _{B}\bar{\ast}\phi ,
\end{equation*}%
on basic $r$-forms $\phi $.
\end{proposition}

Again, we note that $\partial _{T}^{\ast }$ is the holomorphic divergence on
the local quotient manifolds in the foliation charts, and it only depends on
the transverse metric and holomorphic structure.

Thus,%
\begin{equation*}
\partial _{T}^{\ast 2}=0.
\end{equation*}%
Also, since $\partial _{B}$ is the same as the holomorphic differential on
the local quotient manifold in the foliation charts,%
\begin{equation*}
\partial _{B}^{2}=0.
\end{equation*}%
Taking adjoints with respect to basic forms, from the three equations above
we have%
\begin{equation*}
\partial _{T}:=\left( \partial _{T}^{\ast }\right) ^{\ast }=\partial
_{B}-\epsilon \left( \kappa _{B}^{1,0}\right)
\end{equation*}%
\begin{equation*}
\partial _{T}^{2}=0,~\left( \partial _{B}^{\ast }\right) ^{2}=0.
\end{equation*}%
Then 
\begin{equation}
\partial _{T}\left( 1\right) =\left( \partial _{B}-\epsilon (\kappa
_{B}^{1,0})\right) \left( 1\right) =-\kappa _{B}^{1,0}  \label{deltaT(1)}
\end{equation}%
and%
\begin{equation*}
\partial _{B}\kappa _{B}^{1,0}=\left( \partial _{T}+\epsilon (\kappa
_{B}^{1,0})\right) \kappa _{B}^{1,0}=\partial _{T}\kappa _{B}^{1,0}=\partial
_{T}\left( -\partial _{T}1\right) =0.
\end{equation*}%
It also follows that%
\begin{equation*}
\partial _{T}\kappa _{B}^{1,0}=0.
\end{equation*}

Similarly, 
\begin{equation*}
\bar{\partial}_{B}\kappa _{B}^{0,1}=\bar{\partial}_{T}\kappa _{B}^{0,1}=0.
\end{equation*}%
Since $d\kappa _{B}=0$, we have%
\begin{equation*}
0=\left( \partial _{B}+\bar{\partial}_{B}\right) \left( \kappa
_{B}^{1,0}+\kappa _{B}^{0,1}\right) =\partial _{B}\kappa _{B}^{0,1}+\bar{%
\partial}_{B}\kappa _{B}^{1,0}.
\end{equation*}

So 
\begin{equation*}
\func{Re}\left( \partial _{B}\kappa _{B}^{0,1}\right) =0.
\end{equation*}%
We summarize these results in the following Proposition.

\begin{proposition}
\label{holomorphicKappaClosedProp}Let $\left( M,\mathcal{F},J,g_{Q}\right) $
be a foliation with a holonomy-invariant transverse complex structure and
transverse Hermitian metric. Given a bundle-like metric for $\left( M,%
\mathcal{F}\right) $ compatible with the Hermitian metric, let $\kappa
_{B}=\kappa _{B}^{1,0}+\kappa _{B}^{0,1}$ be the basic component of the mean
curvature one-form, with $\kappa _{B}^{1,0}=\frac{1}{2}\left( \kappa
_{B}+iJ\kappa _{B}\right) \in \Omega _{B}^{1,0}$, $\kappa _{B}^{0,1}=%
\overline{\kappa _{B}^{1,0}}$. Then%
\begin{eqnarray*}
\partial _{B}\kappa _{B}^{1,0} &=&\bar{\partial}_{B}\kappa _{B}^{0,1}=0;~%
\func{Re}\left( \partial _{B}\kappa _{B}^{0,1}\right) =0; \\
\partial _{T}\kappa _{B}^{1,0} &=&\bar{\partial}_{T}\kappa _{B}^{0,1}=0
\end{eqnarray*}
\end{proposition}

We do not expect that $\partial _{B}\kappa _{B}^{0,1}=-\bar{\partial}%
_{B}\kappa _{B}^{1,0}\in \Omega _{B}^{1,1}\left( \mathcal{F}\right) $ would
be in general zero for any metric. Consider Example \ref%
{KaehlerExactMeanCurv}. In the next section we will examine this form more
closely.

\begin{theorem}
\label{holomorphicAlvarezClass}\vspace{0in}Let $\left( M,\mathcal{F}%
,J,g_{Q}\right) $ be a foliation with compact leaf closures with a
holonomy-invariant transverse complex structure and transverse Hermitian
metric. Given a bundle-like metric for $\left( M,\mathcal{F}\right) $
compatible with the Hermitian metric, let $\kappa _{B}=\kappa
_{B}^{1,0}+\kappa _{B}^{0,1}$ be the basic component of the mean curvature
one-form, with $\kappa _{B}^{1,0}\in \Omega _{B}^{1,0}$, $\kappa _{B}^{0,1}=%
\overline{\kappa _{B}^{1,0}}$. Then the cohomology classes of $\kappa
_{B}^{1,0}$ and $\kappa _{B}^{0,1}$ in $H_{\partial _{B}}^{1,0}\left( 
\mathcal{F}\right) $ and $H_{\bar{\partial}_{B}}^{0,1}\left( \mathcal{F}%
\right) $, respectively, are invariant with respect to the choice of
bundle-like metric and transverse metric compatible with the holomorphic
structure.
\end{theorem}

\begin{proof}
By \cite{Al} (Proposition \ref{AlvarezClassThm}), any change of compatible
bundle-like metric and transverse metric changes $\kappa _{B}$ to $\kappa
_{B}^{\prime }=\kappa _{B}+df$ for some basic function $f$. Then $\left(
\kappa _{B}^{1,0}\right) ^{\prime }=\frac{1}{2}\left( \kappa _{B}+iJ\kappa
_{B}\right) +\frac{1}{2}\left( df+iJdf\right) $. Using local coordinates,
one can show that on real-valued basic functions $\partial _{B}=\Pi ^{1,0}d$
where $\Pi ^{1,0}:\Omega _{B}^{1}\left( \mathcal{F}\right) \rightarrow
\Omega _{B}^{1,0}\left( \mathcal{F}\right) $ is the projection $\alpha
\mapsto \frac{1}{2}\left( \alpha +iJ\alpha \right) $, we have $\left( \kappa
_{B}^{1,0}\right) ^{\prime }=\kappa _{B}^{1,0}+\partial _{B}f$. Thus $\left[
\kappa _{B}^{1,0}\right] \in H_{\partial _{B}}^{1,0}\left( \mathcal{F}%
\right) $ is independent of the choice of the metric choices. The analogous
proof for $\kappa _{B}^{0,1}$ is similar.
\end{proof}

\begin{remark}
As in Corollary \ref{harmonicMeanCurvCorollary}, we can may multiply the
metric along the leaves by a conformal factor to yield any possible element $%
\left( \kappa _{B}^{1,0}\right) ^{\prime }\in \left[ \kappa _{B}^{1,0}\right]
$.
\end{remark}

\begin{remark}
Because $\kappa _{B}^{1,0}=\frac{1}{2}\left( \kappa _{B}+iJ\kappa
_{B}\right) $ and $\kappa _{B}=\kappa _{B}^{1,0}+\overline{\kappa _{B}^{1,0}}
$, $M$ is taut (i.e. $\kappa _{B}$ is $d_{B}$-exact) if and only if $\left[
\kappa _{B}^{1,0}\right] $ is trivial.
\end{remark}

\begin{proposition}
\label{DomGen}Let $\left( M,\mathcal{F},J,g_{Q}\right) $ be a foliation with
compact leaf closures with a holonomy-invariant transverse complex structure
and transverse Hermitian metric. Then there exists a bundle-like metric
compatible with the transverse Hermitian metric, such that $\kappa ^{1,0}$
and $\kappa ^{0,1}$ are basic forms; that is, $\kappa ^{1,0}=\kappa
_{B}^{1,0}$ and $\kappa ^{0,1}=\kappa _{B}^{0,1}$.
\end{proposition}

\begin{proof}
By \cite{Dom} (Proposition \ref{tensenessTheorem}), there exists a
bundle-like metric that does not change the transverse structure such that $%
\kappa =\kappa _{B}$. Now apply the projections $\Pi ^{1,0}$ and $\Pi ^{0,1}$%
.
\end{proof}

Let%
\begin{equation*}
\square _{B}=\partial _{B}^{\ast }\partial _{B}+\partial _{B}\partial
_{B}^{\ast },~\overline{\square }_{B}=\bar{\partial}_{B}^{\ast }\bar{\partial%
}_{B}+\bar{\partial}_{B}\bar{\partial}_{B}^{\ast }
\end{equation*}%
On a closed manifold $M$, it is clear that a basic form $\alpha $ satisfies $%
\square _{B}\alpha =0$ if and only if $\partial _{B}\alpha =0$ and $\partial
_{B}^{\ast }\alpha =0$. Also, if $\alpha \in \Omega _{B}^{r,0}\left( 
\mathcal{F}\right) $, automatically $\bar{\partial}_{B}^{\ast }\alpha =0$,
so that $\partial _{B}^{\ast }\alpha =0$ implies that $\delta _{B}\alpha
=\left( \partial _{B}^{\ast }+\bar{\partial}_{B}^{\ast }\right) \alpha =0$,
where $\delta _{B}$ is the adjoint of $d$ restricted to basic complex-valued
forms.

\begin{lemma}
\label{realOpLemma}Let $\left( M,\mathcal{F},J,g_{Q}\right) $ be a
transverse K\"{a}hler foliation on a closed manifold. Then for any
complex-valued basic function $f$, 
\begin{equation*}
\delta _{T}df=2\partial _{T}^{\ast }\partial _{B}f=2\bar{\partial}_{T}^{\ast
}\bar{\partial}_{B}f.
\end{equation*}%
In particular, the operator $\partial _{T}^{\ast }\partial _{B}$ is a real
operator on functions.
\end{lemma}

\begin{proof}
Since $\delta _{T}$, $\partial _{T}^{\ast }$ and $\bar{\partial}_{T}^{\ast }$
correspond to the divergences $d^{\ast }$, $\partial ^{\ast }$, $\bar{%
\partial}^{\ast }$ on the local quotient manifold, this Lemma follows
directly from the local fact that $\Delta _{d}=2\Delta _{\partial }=2\Delta
_{\bar{\partial}}$ on K\"{a}hler manifolds.
\end{proof}

\vspace{0in}With our string of successes of projecting using $\Pi ^{r,s}$,
one would hope that an analogue of Proposition \ref{harmonicMeanCurvThm} can
be found just as easily. However, as the following remark shows, we are not
so lucky.

\begin{remark}
Let $\left( M,\mathcal{F},J,g_{Q}\right) $ be a foliation on a closed
manifold with a holonomy-invariant transverse complex structure and
transverse Hermitian metric. By Proposition \ref{harmonicMeanCurvThm},
choose $g_{M}$ to be a bundle-like metric on $M$ compatible with the
transverse Hermitian structure and chosen so that the mean curvature $\kappa 
$ is basic harmonic --- that is, so that $\kappa =\kappa _{B}$, $\delta
_{B}\kappa =0$. Then observe that%
\begin{eqnarray*}
0 &=&\delta _{B}\kappa \\
&=&\left( \partial _{B}^{\ast }+\bar{\partial}_{B}^{\ast }\right) \left(
\kappa ^{1,0}+\kappa ^{0,1}\right) \\
&=&\partial _{B}^{\ast }\kappa ^{1,0}+\overline{\partial _{B}^{\ast }\kappa
^{1,0}}=2\func{Re}\left( \partial _{B}^{\ast }\kappa ^{1,0}\right) .
\end{eqnarray*}%
Hence, Proposition \ref{harmonicMeanCurvThm} gives us no control over the
imaginary part of $\partial _{B}^{\ast }\kappa ^{1,0}$. On the other hand,
suppose we are able to find a bundle-like metric $g_{M}$ such that 
\begin{equation*}
\partial _{B}^{\ast }\kappa ^{1,0}=0.
\end{equation*}%
Then by the calculation above, $0=\delta _{B}\kappa $, so in fact $\kappa $
is basic harmonic. However, by Corollary \ref{harmonicMeanCurvCorollary},
the leafwise volume form $\chi _{\mathcal{F}}$ is determined up to a
constant scale factor, so the form $\kappa $ is uniquely determined.
\end{remark}

From the discussion in the remark above we at least have the following.

\begin{proposition}
Let $\left( M,\mathcal{F},J,g_{Q}\right) $ be a foliation on a closed
manifold with a holonomy-invariant transverse complex structure and
transverse Hermitian metric. Suppose that there exists a bundle-like metric
compatible with the transverse structure such that $\partial _{B}^{\ast
}\kappa ^{1,0}=0$. Then the mean curvature $\kappa $ is basic harmonic, and $%
\kappa $ is the unique element of $\left[ \kappa \right] \in $ $%
H_{B}^{1}\left( \mathcal{F}\right) $ with this property.
\end{proposition}

If $\left( M,\mathcal{F},J,g_{Q}\right) $ is transversely K\"{a}hler, then
the situation of the previous proposition always occurs.

\begin{proposition}
\label{harmonicMeanCurvKahlerThm}Let $\left( M,\mathcal{F},J,g_{Q}\right) $
be a transverse K\"{a}hler foliation on a closed manifold. Then there exists
a bundle-like metric compatible with the K\"{a}hler structure such that $%
\kappa $ is basic harmonic; that is, $\kappa =\kappa _{B}$ and $\delta
_{B}\kappa =0$. For that same metric,%
\begin{equation*}
\partial _{B}^{\ast }\kappa ^{1,0}=\delta _{B}\kappa ^{1,0}=0,
\end{equation*}%
so that also 
\begin{equation*}
\square _{B}\kappa ^{1,0}=0,~\bar{\partial}_{B}^{\ast }\kappa ^{0,1}=\delta
_{B}\kappa ^{0,1}=0,~\overline{\square }_{B}\kappa ^{0,1}=0.
\end{equation*}
\end{proposition}

\begin{proof}
Let the bundle-like metric be chosen as in Proposition \ref%
{harmonicMeanCurvThm}, so that $\kappa =\kappa _{B}$ and $\delta _{B}\kappa
=0$. Since the foliation is transversely K\"{a}hler, $\partial _{T}^{\ast
}\partial _{B}$ on functions is a real operator, as is its adjoint $\partial
_{B}^{\ast }\partial _{T}$, by Lemma \ref{realOpLemma}. But then%
\begin{equation*}
\partial _{B}^{\ast }\partial _{T}\left( 1\right) =-\partial _{B}^{\ast
}\kappa _{B}^{1,0}
\end{equation*}%
is a real-valued function, so that%
\begin{eqnarray*}
\partial _{B}^{\ast }\kappa _{B}^{1,0} &=&\frac{1}{2}\left( \partial
_{B}^{\ast }\kappa _{B}^{1,0}+\overline{\partial _{B}^{\ast }\kappa
_{B}^{1,0}}\right) \\
&=&\frac{1}{2}\left( \partial _{B}^{\ast }\kappa _{B}^{1,0}+\bar{\partial}%
_{B}^{\ast }\kappa _{B}^{0,1}\right) \\
&=&\frac{1}{2}\left( \delta _{B}\kappa _{B}^{1,0}+\delta _{B}\kappa
_{B}^{0,1}\right) =\frac{1}{2}\delta _{B}\kappa _{B}=0.
\end{eqnarray*}
\end{proof}

\begin{remark}
After examining the proof in \cite{MO} (Proposition \ref{harmonicMeanCurvThm}%
), it does not appear that Proposition \ref{harmonicMeanCurvKahlerThm} is
true in the more general transverse Hermitian foliation case. Finding such a
metric is tantamount to finding a smooth, positive basic function $\psi $
such that 
\begin{equation*}
\partial _{B}^{\ast }\partial _{T}\psi =0.
\end{equation*}%
From the original proof, there is a $\psi $ that is unique up to a
multiplicative constant such that 
\begin{equation*}
\func{Re}\left( \partial _{B}^{\ast }\partial _{T}\psi \right) =0.
\end{equation*}%
However, it seems unlikely that the imaginary part would also vanish for
this $\psi $.
\end{remark}

For future use, we recall the Hodge theorem for basic Dolbeault cohomology.

\begin{theorem}
(Proved in \cite[Th\'eor\`eme 3.3.3]{EK} for general tranversely Hermitian
foliations on a compact manifold, stated in general in \cite[Theorem 3.21]%
{JR2}) \label{DolbeaultHodge}Let $(M,\mathcal{F},J,g_{Q})$ be a transverse K%
\"{a}hler foliation on a compact Riemannian manifold with bundle-like metric 
$g_{M}$. Then 
\begin{equation*}
\Omega _{B}^{r,s}(\mathcal{F})\cong \mathcal{H}_{B}^{r,s}\oplus \mathrm{Im}%
\bar{\partial}_{B}\oplus \mathrm{Im}\bar{\partial}_{B}^{\ast },
\end{equation*}%
where $\mathcal{H}_{B}^{r,s}=\mathrm{Ker}\overline{\square }_{B}$ is finite
dimensional. Moreover, $\mathcal{H}_{B}^{r,s}\cong H_{B}^{r,s}$.
\end{theorem}

\section{Another class in basic Dolbeault cohomology\label{newClassSection}}

Let $\left( M,\mathcal{F},J,g_{Q}\right) $ be a foliation with compact leaf
closures with a holonomy-invariant transverse complex structure and
transverse Hermitian metric. By Theorem \ref{holomorphicAlvarezClass}, the
cohomology classes of $\kappa _{B}^{1,0}$ and $\kappa _{B}^{0,1}$ in $%
H_{\partial _{B}}^{1,0}\left( \mathcal{F}\right) $ and $H_{\bar{\partial}%
_{B}}^{0,1}\left( \mathcal{F}\right) $, respectively, are invariant with
respect to the choice of compatible bundle-like metric. Observe that we may
obtain an additional invariant basic Dolbeault cohomology class from the
transverse holomorphic structure. Note that for any transversely holomorphic
foliation, $\partial _{B}\bar{\partial}_{B}=-\bar{\partial}_{B}\partial _{B}$%
, and $\left( \partial _{B}\bar{\partial}_{B}\right) ^{2}=0$, so that 
\begin{equation*}
\Omega ^{0,0}\left( \mathcal{F}\right) \overset{\partial _{B}\bar{\partial}%
_{B}}{\rightarrow }\Omega ^{1,1}\left( \mathcal{F}\right) \overset{\partial
_{B}\bar{\partial}_{B}}{\rightarrow }\Omega ^{2,2}\left( \mathcal{F}\right) 
\overset{\partial _{B}\bar{\partial}_{B}}{\rightarrow }...\overset{\partial
_{B}\bar{\partial}_{B}}{\rightarrow }\Omega ^{n,n}\left( \mathcal{F}\right)
\end{equation*}%
forms a differential complex, and so that the cohomology $H_{\partial _{B}%
\bar{\partial}_{B}}^{j,j}\left( \mathcal{F}\right) $ is well-defined.

Also, observe that, with $\ast $ denoting the adjoint with respect to basic
forms, 
\begin{equation*}
\Delta _{\partial _{B}\bar{\partial}_{B}}:=\left( \partial _{B}\bar{\partial}%
_{B}\right) ^{\ast }\partial _{B}\bar{\partial}_{B}+\partial _{B}\bar{%
\partial}_{B}\left( \partial _{B}\bar{\partial}_{B}\right) ^{\ast }.
\end{equation*}%
Also, for $\varphi \in $ $\Omega _{B}^{j,j}\left( \mathcal{F}\right) $, $%
\Delta _{\partial _{B}\bar{\partial}_{B}}\varphi =0$ if and only if 
\begin{equation*}
\partial _{B}\bar{\partial}_{B}\varphi =0\text{ and }\partial _{B}^{\ast }%
\bar{\partial}_{B}^{\ast }\varphi =0.
\end{equation*}%
We get the resulting Hodge theorem, since $\Delta _{\partial _{B}\bar{%
\partial}_{B}}$ is strongly elliptic on basic forms on a Riemannian
foliation.

\begin{proposition}
\label{deldelbarHodgeTheoremProp}Let $(M,\mathcal{F},J,g_{Q})$ be a
transverse Hermitian foliation on a compact Riemannian manifold with
bundle-like metric $g_{M}$. Then 
\begin{equation*}
\Omega _{B}^{r,r}(\mathcal{F})\cong \mathcal{H}_{\partial _{B}\bar{\partial}%
_{B}}^{r,r}\oplus \mathrm{Im}\partial _{B}\bar{\partial}_{B}\oplus \mathrm{Im%
}\partial _{B}^{\ast }\bar{\partial}_{B}^{\ast },
\end{equation*}%
where $\mathcal{H}_{\partial _{B}\bar{\partial}_{B}}^{r,r}=\ker \left(
\left. \Delta _{\partial _{B}\bar{\partial}_{B}}\right\vert _{\Omega
^{r,r}\left( \mathcal{F}\right) }\right) \cong H_{\partial _{B}\bar{\partial}%
_{B}}^{r,r}\left( \mathcal{F}\right) $ is finite dimensional.
\end{proposition}

\begin{remark}
The $\Delta _{\partial _{B}\bar{\partial}_{B}}$-harmonic form
representatives in the classes in $H_{\partial _{B}\bar{\partial}%
_{B}}^{r,r}\left( \mathcal{F}\right) $ are precisely those with minimum $%
L^{2}$-norm. The usual proof works: if $\alpha $ is one such $\Delta
_{\partial _{B}\bar{\partial}_{B}}$-harmonic $\left( r,r\right) $-form and $%
\beta \in \Omega _{B}^{r-1,r-1}\left( \mathcal{F}\right) $, then%
\begin{eqnarray*}
\ll \alpha +\partial _{B}\bar{\partial}_{B}\beta ,\alpha +\partial _{B}\bar{%
\partial}_{B}\beta \gg &=&\ll \alpha ,\alpha \gg +2\func{Re}\ll \partial _{B}%
\bar{\partial}_{B}\beta ,\alpha \gg +\ll \partial _{B}\bar{\partial}%
_{B}\beta ,\partial _{B}\bar{\partial}_{B}\beta \gg \\
&=&\ll \alpha ,\alpha \gg +2\func{Re}\ll \beta ,\left( \partial _{B}\bar{%
\partial}_{B}\right) ^{\ast }\alpha \gg +\ll \partial _{B}\bar{\partial}%
_{B}\beta ,\partial _{B}\bar{\partial}_{B}\beta \gg \\
&=&\ll \alpha ,\alpha \gg +\ll \partial _{B}\bar{\partial}_{B}\beta
,\partial _{B}\bar{\partial}_{B}\beta \gg \,\geq \,\ll \alpha ,\alpha \gg .
\end{eqnarray*}
\end{remark}

\begin{theorem}
\label{newClassDefinedTheorem}Let $\left( M,\mathcal{F},J,g_{Q}\right) $ be
a foliation with compact leaf closures with a holonomy-invariant transverse
complex structure and transverse Hermitian metric. For a given compatible
bundle-like metric, let $\kappa _{B}^{1,0}$ and $\kappa _{B}^{0,1}$ be the
corresponding basic components of the mean curvature $1$-form $\kappa $.
Then the form $\partial _{B}\kappa _{B}^{0,1}$ is $\partial _{B}\bar{\partial%
}_{B}$-closed, and its cohomology class in $H_{\partial _{B}\bar{\partial}%
_{B}}^{1,1}\left( \mathcal{F}\right) $ is invariant of the choice of
transverse metric and bundle-like metric. A similar result is true for $%
\left[ \bar{\partial}_{B}\kappa _{B}^{1,0}\right] =-\left[ \partial
_{B}\kappa _{B}^{0,1}\right] \in H_{\partial _{B}\bar{\partial}%
_{B}}^{1,1}\left( \mathcal{F}\right) $.
\end{theorem}

\begin{proof}
Since $0=d_{B}^{2}=\left( \partial _{B}+\bar{\partial}_{B}\right) ^{2}$, $%
\bar{\partial}_{B}\left( \partial _{B}\kappa _{B}^{0,1}\right) =-\partial
_{B}\left( \bar{\partial}_{B}\kappa _{B}^{0,1}\right) =0$, so $\partial
_{B}\kappa _{B}^{0,1}$ is $\bar{\partial}_{B}$-closed and thus $\partial _{B}%
\bar{\partial}_{B}$-closed. By \ref{holomorphicAlvarezClass}, any other
choice of compatible transverse metric and bundle-like metric yields $\left(
\kappa _{B}^{0,1}\right) ^{\prime }=\kappa _{B}^{0,1}+\bar{\partial}_{B}f$
for some complex-valued function $f$. Then%
\begin{eqnarray*}
\partial _{B}\left( \kappa _{B}^{0,1}\right) ^{\prime } &=&\partial
_{B}\kappa _{B}^{0,1}+\partial _{B}\bar{\partial}_{B}f \\
&=&\partial _{B}\kappa _{B}^{0,1}+\bar{\partial}_{B}\partial _{B}\left(
-f\right) .
\end{eqnarray*}%
Applying conjugation, we get a similar result for $\bar{\partial}_{B}\kappa
_{B}^{1,0}$.
\end{proof}

\begin{remark}
By Proposition \ref{holomorphicKappaClosedProp}, $\partial _{B}\kappa
_{B}^{0,1}$ and $\bar{\partial}_{B}\kappa _{B}^{1,0}$ are pure imaginary
forms.
\end{remark}

\begin{remark}
If $\left( M,\mathcal{F}\right) $ as above is a taut foliation, i.e. $\kappa
_{B}$ is $d$-exact, then the class $\left[ \partial _{B}\kappa _{B}^{0,1}%
\right] $ is trivial because $\kappa _{B}^{0,1}$ is $\bar{\partial}_{B}$%
-exact. However, the converse is false; see Example \ref%
{ExampleNewClassZeroButNotTaut}.
\end{remark}

\begin{remark}
\label{realizationOfClassRep}It is clear from the proof of Theorem \ref%
{newClassDefinedTheorem} that we may modify the metric along the leaves of
the foliation so that any purely imaginary element of the class $\left[
\partial _{B}\kappa _{B}^{0,1}\right] \in H_{\partial _{B}\bar{\partial}%
_{B}}^{1,1}\left( \mathcal{F}\right) $ may be realized, since by multiplying
the leafwise metric by a conformal factor results in an arbitrary
real-valued $f$ such that $\partial _{B}\left( \kappa _{B}^{0,1}\right)
^{\prime }=\partial _{B}\kappa _{B}^{0,1}+\partial _{B}\bar{\partial}_{B}f$.
That is, if for some complex-valued function $h$, $\partial _{B}\left(
\kappa _{B}^{0,1}\right) ^{\prime }=\partial _{B}\kappa _{B}^{0,1}+\partial
_{B}\bar{\partial}_{B}h$, then since $\partial _{B}\left( \kappa
_{B}^{0,1}\right) ^{\prime }$ and $\partial _{B}\kappa _{B}^{0,1}$ are pure
imaginary, we have 
\begin{equation*}
\partial _{B}\bar{\partial}_{B}h=-\overline{\partial _{B}\bar{\partial}_{B}h}%
=-\bar{\partial}_{B}\partial _{B}\overline{h}=\partial _{B}\bar{\partial}_{B}%
\overline{h}=\partial _{B}\bar{\partial}_{B}\left( \func{Re}\left( h\right)
\right) ,
\end{equation*}%
so that $h$ may always be taken to be real.

\begin{proposition}
\label{nontautAndCLassZeroImpliesH1gt1Prop}Let $\left( M,\mathcal{F}%
,J,g_{Q}\right) $ be a nontaut foliation with basic harmonic mean curvature $%
\kappa _{B}$ on a compact manifold with a holonomy-invariant transverse
complex structure and transverse Hermitian metric. Then $\partial _{B}\kappa
_{B}^{0,1}=0$ if and only if $d\left( J\kappa _{B}\right) =0$. If $\partial
_{B}\kappa _{B}^{0,1}=0$, then $\left[ \kappa _{B}\right] $ and $\left[
J\kappa _{B}\right] $ provide two linearly independent basic cohomology
classes, so that $\dim H_{b}^{1}\left( \mathcal{F}\right) \geq 2$.

\begin{proof}
The condition $\partial _{B}\kappa _{B}^{0,1}=0$ is equivalent to 
\begin{equation*}
0=d\kappa _{B}^{0,1}=\frac{1}{2}\left( d\kappa _{B}-idJ\kappa _{B}\right) ,
\end{equation*}%
which is equivalent to $d\left( J\kappa _{B}\right) =0$. Next, suppose that $%
J\kappa _{B}\neq \kappa _{B}$ are in the same cohomology class, so that $%
J\kappa _{B}-\kappa _{B}=df$ for some nonzero exact $1$-form. Observe that $%
\ll \kappa _{B},J\kappa _{B}\gg $ is necessarily zero since $J$ is an
isometry, and also $\ll \kappa _{B},df\gg \, =0$ since $\kappa _{B}$ is
basic harmonic. But then $\kappa _{B}$ is orthogonal to itself because $%
\kappa _{B}=J\kappa _{B}-df$, which is possible only if $\kappa _{B}=0$. The
conclusion follows.
\end{proof}
\end{proposition}
\end{remark}

We will see that the differential form $\partial _{B}\kappa _{B}^{0,1}$ has
particular significance for K\"{a}hler foliations.

\begin{lemma}
\label{delKappa=0Lemma}Let $\left( M,\mathcal{F},J,g_{Q}\right) $ be a
foliation with compact leaf closures, a holonomy-invariant transverse
complex structure and transverse Hermitian metric. Let $V\in \Gamma
_{B}\left( Q^{1,0}\right) $, so that $V^{\flat }\in \Omega _{B}^{0,1}\left( 
\mathcal{F}\right) $. Then the following are equivalent.

\begin{enumerate}
\item $\partial _{B}V^{\flat }=0.$

\item $\partial _{B}\circ V\lrcorner +V\lrcorner \circ \partial _{B}=\nabla
_{V}$ as an operator on locally defined basic differential forms.

\item $\nabla _{Z}V=0$ for all $Z\in \Gamma \left( Q^{1,0}\right) $.

\item Conjugates of the above statements.
\end{enumerate}
\end{lemma}

\begin{proof}
We assume that at the point we are evaluating the operators, the local bases 
$\left\{ V_{a}\right\} $ and $\left\{ \omega _{a}\right\} $ are chosen so
that all covariant derivatives vanish (we can do that because these are
locally basic sections of $Q$ and $Q^{\ast }$). Statement (1) is equivalent
to%
\begin{equation*}
\sum \omega ^{a}\wedge \nabla _{V_{a}}V^{\flat }=0.
\end{equation*}%
Since $\nabla _{V_{a}}V^{\flat }$ is type $\left( 0,1\right) $, this is
equivalent to $\nabla _{V_{a}}V^{\flat }=0$ for all $a$, which is equivalent
to $\nabla _{Z}V^{\flat }=0$ for all $Z\in Q^{1,0}$, i.e. statement (3).
Next, assume (3): Then $\nabla _{V_{a}}V=0$ ($1\leq a\leq n$). Then for any $%
\phi \in \Omega _{B}^{r,s}(\mathcal{F})$, 
\begin{align*}
\partial _{B}\left( V\lrcorner \,\phi \right) & =\sum_{a}\omega ^{a}\wedge
(\nabla _{V_{a}}V)\lrcorner \,\phi +\sum_{a}\omega ^{a}\wedge V\lrcorner
\nabla _{V_{a}}\phi \\
& =-V\lrcorner \,\bar{\partial}_{B}\phi +\nabla _{V}\phi ,
\end{align*}%
since $\omega ^{a}\wedge $ $\left( V\lrcorner \right) +V\lrcorner \left(
\omega ^{a}\wedge \right) =\left\langle V,V_{a}\right\rangle $. Next,
consider (2). If $V$ satisfies (2), then $V=\sum f_{b}V_{b}$ for some basic
functions $f_{b}$. Then for all $b$, 
\begin{eqnarray*}
\partial _{B}V\lrcorner \,\omega ^{b}+V\lrcorner \,\partial _{B}\omega ^{b}
&=&\partial _{B}f_{b}-\sum_{a}f_{b}\omega ^{a}\wedge \left( V_{b}\lrcorner
\right) \nabla _{V_{a}}\omega ^{b}+f_{b}\nabla _{V_{b}}\omega ^{b}=\partial
_{B}f_{b} \\
&=&\sum_{a}\left( V_{a}f_{b}\right) \omega ^{a}=0,
\end{eqnarray*}%
So at the point in question, (2)\ implies that $V_{a}f_{b}=0$ for all $a,b$
at that point. On the other hand, (3)\ is equivalent to 
\begin{eqnarray*}
\nabla _{V_{a}}V &=&\sum \left( V_{a}f_{b}\right) V_{b}+\sum f_{b}\nabla
_{V_{a}}V_{b} \\
&=&\sum \left( V_{a}f_{b}\right) V_{b}=0
\end{eqnarray*}%
for all $a$ if and only if $V_{a}f_{b}$ for all $a,b$ as well.
\end{proof}

\begin{proposition}
\label{FormulasForBoxB}Let $\left( M,\mathcal{F},J,g_{Q}\right) $ be a
foliation with compact leaf closures with a holonomy-invariant transverse
complex structure and transverse Hermitian metric. Then%
\begin{equation*}
\square _{B}=\Delta _{\partial }^{Q}+\partial _{B}\circ H^{1,0}\lrcorner
+H^{1,0}\lrcorner \circ \partial _{B},
\end{equation*}%
where $\Delta _{\partial }^{Q}=\partial _{T}^{\ast }\partial _{B}+\partial
_{B}\partial _{T}^{\ast }$ is the $\partial $-Laplacian on differential
forms on the local quotients of foliation charts. Similarly,%
\begin{equation*}
\overline{\square }_{B}=\Delta _{\bar{\partial}}^{Q}+\bar{\partial}_{B}\circ
H^{0,1}\lrcorner +H^{0,1}\lrcorner \circ \bar{\partial}_{B}
\end{equation*}
\end{proposition}

\begin{proof}
From Proposition \ref{delBstarFormulaProp},%
\begin{equation*}
\square _{B}=\left( \partial _{T}^{\ast }+H^{1,0}\lrcorner \right) \partial
_{B}+\partial _{B}\left( \partial _{T}^{\ast }+H^{1,0}\lrcorner \right) .
\end{equation*}
\end{proof}

\begin{corollary}
\label{KahlerCaseClassVanishesThm}Let $\left( M,\mathcal{F},J,g_{Q}\right) $
be a foliation with compact leaf closures with a holonomy-invariant
transverse complex structure, transverse Hermitian metric, and compatible
bundle-like metric. Then the following are equivalent:

\begin{enumerate}
\item $\square _{B}=\Delta _{\partial }^{Q}+\nabla _{H^{1,0}}$ as operators
on locally defined basic differential forms.

\item $\overline{\square }_{B}=\Delta _{\bar{\partial}}^{Q}+\nabla
_{H^{0,1}} $ as operators on locally defined basic differential forms.

\item $\partial _{B}\kappa _{B}^{0,1}=0.$
\end{enumerate}
\end{corollary}

\begin{proof}
Apply Lemma \ref{delKappa=0Lemma} to $V=H^{1,0}$, $V^{\flat }=\kappa
_{B}^{0,1}$ and Proposition \ref{FormulasForBoxB}.
\end{proof}

\begin{remark}
The condition $\partial_B\kappa_B^{0,1}=0$ is equivalent to the condition
that if we write 
\begin{equation*}
H^{1,0}=\sum_j H_j^{1,0}(z)\partial_{z_j}
\end{equation*}
in local transverse coordinates, then the functions $H_j^{1,0}(z) $ are
antiholomorphic.
\end{remark}

\begin{theorem}
\label{classTrivialTheorem}Let $\left( M,\mathcal{F},J,g_{Q}\right) $ be a
transverse K\"{a}hler foliation with compatible bundle-like metric. If the
class $\left[ \partial _{B}\kappa _{B}^{0,1}\right] \in H_{\partial _{B}\bar{%
\partial}_{B}}^{1,1}\left( \mathcal{F}\right) $ is trivial, there exists a
bundle-like metric such that as operators on basic differential forms, 
\begin{equation*}
\square _{B}=\overline{\square }_{B}-i\nabla _{JH}.
\end{equation*}
\end{theorem}

\begin{proof}
If $\left( M,\mathcal{F},J,g_{Q}\right) $ is a transverse K\"{a}hler
foliation with compatible bundle-like metric so that $\Delta _{\bar{\partial}%
}^{Q}=\Delta _{\partial }^{Q}$, the class $\left[ \partial _{B}\kappa
_{B}^{0,1}\right] \in H_{\partial _{B}\bar{\partial}_{B}}^{1,1}\left( 
\mathcal{F}\right) $ is trivial if and only if there exists a compatible
bundle-like metric such that $\partial _{B}\kappa _{B}^{0,1}=0$. By
Corollary \ref{KahlerCaseClassVanishesThm}, $\partial _{B}\kappa
_{B}^{0,1}=0 $ implies that on basic forms, 
\begin{eqnarray*}
\square _{B} &=&\Delta _{\partial }^{Q}+\nabla _{H^{1,0}} \\
&=&\Delta _{\bar{\partial}}^{Q}+\nabla _{H^{1,0}} \\
&=&\Delta _{\bar{\partial}}^{Q}+\nabla _{H^{0,1}}+\nabla _{H^{1,0}-H^{0,1}}
\\
&=&\overline{\square }_{B}+\nabla _{-iJH}.
\end{eqnarray*}
\end{proof}

\begin{corollary}
\label{DolbeaultSymmetryCorollary}Let $\left( M,\mathcal{F},J,g_{Q}\right) $
be a transverse K\"{a}hler foliation with compatible bundle-like metric.
Suppose that the $\partial _{B}\kappa _{B}^{0,1}=0$. Then%
\begin{eqnarray*}
\ker \left( \left. \overline{\square }_{B}\right\vert _{\Omega
_{B}^{r,s}}\right) \cap \ker \left( \left. \square _{B}\right\vert _{\Omega
_{B}^{r,s}}\right) &=&\ker \left( \left. \overline{\square }_{B}\right\vert
_{\Omega _{B}^{r,s}}\right) \cap \ker \left( \left. \nabla _{JH}\right\vert
_{\Omega _{B}^{r,s}}\right) \\
&=&\ker \left( \left. \square _{B}\right\vert _{\Omega _{B}^{r,s}}\right)
\cap \ker \left( \left. \nabla _{JH}\right\vert _{\Omega _{B}^{r,s}}\right) .
\end{eqnarray*}
\end{corollary}

\begin{proof}
This follows directly from Theorem \ref{classTrivialTheorem}.
\end{proof}

Note that the hypothesis on $\partial _{B}\kappa _{B}^{0,1}$ is needed, as
Example \ref{KaehlerExactMeanCurv} shows.

\section{Lefschetz decompositions\label{LefschetzDecompSection}}

We begin with some notation. Let $\left( M,\mathcal{F},J,g_{Q}\right) $ be a
transverse K\"{a}hler foliation with compatible bundle-like metric, with
associated K\"{a}hler form $\omega $. Let $L:\Omega _{B}^{r}(\mathcal{F}%
)\rightarrow \Omega _{B}^{r+2}(\mathcal{F})$ and $\Lambda :\Omega _{B}^{r}(%
\mathcal{F})\rightarrow \Omega _{B}^{r-2}(\mathcal{F})$ be given by 
\begin{equation*}
L(\phi )=\omega \wedge \phi ,\quad \Lambda (\phi )=\omega \lrcorner \,\phi ,
\end{equation*}%
respectively, where $(\beta _{1}\wedge \beta _{2})\lrcorner =\beta
_{2}^{\sharp }\lrcorner \beta _{1}^{\sharp }\lrcorner $ for any basic
1-forms $\beta _{i}(i=1,2)$. 
%Since the K\"ahler form $\omega\in\Omega_B^{1,1}(\mathcal F)$ is of type $(1,1)$, $L$ and $\Lambda$ preserve the types.
It follows that $\langle L\phi ,\psi \rangle =\langle \phi ,\Lambda \psi
\rangle $ and $\Lambda =(-1)^{j}\bar{\ast}L\bar{\ast}$ on basic $j$-forms.
For $X\in Q$, from \cite{JJ} we have 
\begin{align}
& [L,X\lrcorner ]=\epsilon (JX^{b}),\quad \lbrack \Lambda ,\epsilon
(X^{b})]=-(JX)\lrcorner ,  \label{LCommutatorFormulas} \\
& [L,\epsilon (X^{b})]=[\Lambda ,X\lrcorner ]=0.  \notag
\end{align}%
The formulas above extend is exactly the same way to complex vectors $X$. We
extend the complex structure $J$ to $\Omega _{B}^{r}(\mathcal{F})$ by the
formula 
\begin{equation*}
J\phi =\sum_{a=1}^{2n}J\theta ^{a}\wedge E_{a}\lrcorner \,\phi .
\end{equation*}%
This formula is consistent with $(J\theta )(X)=-\theta (JX)$ for one-forms $%
\theta $, and for instance $\left( Jv\right) ^{\flat }=Jv^{\flat }$ for
vectors $v$. The operator $J:\Omega _{B}^{r,s}(\mathcal{F})\rightarrow
\Omega _{B}^{r,s}(\mathcal{F})$ is skew-Hermitian: $\langle J\phi ,\psi
\rangle +\langle \phi ,J\psi \rangle =0$, and $J\phi =i\left( s-r\right)
\phi $ for any $\phi \in \Omega _{B}^{r,s}(\mathcal{F})$. This is not the
same as the operator $C$ induced from the pullback $J^{\ast }$ used often in
K\"{a}hler geometry.

%By using (3.19)$\sim$(3.22), we have the followings.
We quote some known results as follows.

\begin{proposition}
\cite[Proposition 3.3]{JJ} \label{LProp}If $(M,\mathcal{F},J,g_{Q})$ is a
transverse K\"{a}hler foliation on a compact Riemannian manifold with
bundle-like metric $g_{M}$, 
\begin{equation*}
\lbrack L,J]=[\Lambda ,J]=[L,d_{B}]=[\Lambda ,\delta _{B}]=0.
\end{equation*}
\end{proposition}

\begin{corollary}
\cite[Proposition 3.4]{JJ}\label{LCor}With the same hypothesis, 
\begin{align}
& [L,\partial _{B}]=[L,\bar{\partial}_{B}]=[\Lambda ,\partial _{B}^{\ast
}]=[\Lambda ,\bar{\partial}_{B}^{\ast }]=0,  \label{LCommutesWithDiffs} \\
& [L,\partial _{B}^{\ast }]=-i\bar{\partial}_{T},\ [L,\bar{\partial}%
_{B}^{\ast }]=i\partial _{T},\ [\Lambda ,\partial _{B}]=-i\bar{\partial}%
_{T}^{\ast },\ [\Lambda ,\bar{\partial}_{B}]=i\partial _{T}^{\ast }.
\label{LCommutesWithCodiffs}
\end{align}
\end{corollary}

%{\bf Proof.} The proofs are trivial from Proposition 3.6. $\Box$

\begin{remark}
All equations above in Proposition~\ref{LProp} and Corollary~\ref{LCor}
continue to hold if we exchange the operators $(\cdot )_{B}$ and $(\cdot
)_{T}$. These results were shown in \cite[Lemma 3.4.4]{EK} in the minimal
foliation case, when $(\cdot )_{B}=(\cdot )_{T}$.
\end{remark}

\begin{proposition}
\label{LaplCommutesWithLProp}If $(M,\mathcal{F},J,g_{Q})$ is a transverse K%
\"{a}hler foliation on a compact Riemannian manifold with bundle-like metric 
$g_{M}$, we have%
\begin{equation*}
\left[ \overline{\square }_{B},L\right] =i~\epsilon \left( \bar{\partial}%
_{B}\kappa _{B}^{1,0}\right)
\end{equation*}%
as operators on basic forms. Similarly,%
\begin{equation*}
\left[ \square _{B},L\right] =i~\epsilon \left( \bar{\partial}_{B}\kappa
_{B}^{1,0}\right) .
\end{equation*}
\end{proposition}

\begin{proof}
By the corollary above, we have 
\begin{eqnarray*}
\overline{\square }_{B}L &=&\bar{\partial}_{B}^{\ast }\bar{\partial}_{B}L+%
\bar{\partial}_{B}\bar{\partial}_{B}^{\ast }L \\
&=&\bar{\partial}_{B}^{\ast }L\bar{\partial}_{B}+\bar{\partial}_{B}L\bar{%
\partial}_{B}^{\ast }-i\bar{\partial}_{B}\bar{\partial}_{T} \\
&=&L\bar{\partial}_{B}^{\ast }\bar{\partial}_{B}+L\bar{\partial}_{B}\bar{%
\partial}_{B}^{\ast }-i\bar{\partial}_{B}\partial _{T}-i\partial _{T}\bar{%
\partial}_{B} \\
&=&L\overline{\square }_{B}-i\left( \bar{\partial}_{B}\partial _{B}+\partial
_{B}\bar{\partial}_{B}-\bar{\partial}_{B}\epsilon \left( \kappa
_{B}^{1,0}\right) -\epsilon \left( \kappa _{B}^{1,0}\right) \bar{\partial}%
_{B}\right) \\
&=&L\overline{\square }_{B}+i\epsilon \left( \bar{\partial}_{B}\kappa
_{B}^{1,0}\right) .
\end{eqnarray*}%
The second part follows from noticing that $\bar{\partial}_{B}\kappa
_{B}^{1,0}$ is pure imaginary and from taking conjugates.
\end{proof}

\begin{lemma}
We have the following identities.

\begin{enumerate}
\item $\left[ \Lambda ,L\right] =\sum \left( n-r\right) P_{r}$ as an
operator on basic forms, where $P_{r}:$ $\Omega _{B}^{\ast }\left( \mathcal{F%
}\right) \rightarrow \Omega _{B}^{r}\left( \mathcal{F}\right) $ is the
projection.

\item $\left[ \sum \left( n-r\right) P_{r},\Lambda \right] =2\Lambda $.

\item $\left[ \sum \left( n-r\right) P_{r},L\right] =-2L$.
\end{enumerate}
\end{lemma}

\begin{proof}
If $\alpha \in \Omega _{B}^{r}\left( \mathcal{F}\right) $, 
\begin{equation*}
\left[ \Lambda ,L\right] \alpha =\sum_{a,b=1}^{n}JE_{b}\lrcorner
E_{b}\lrcorner \theta ^{a}\wedge J\theta ^{a}\wedge \alpha -\theta
^{a}\wedge J\theta ^{a}\wedge JE_{b}\lrcorner E_{b}\lrcorner \alpha \cdot
\end{equation*}%
It is easy to see that for a simple $r$ form $\alpha =\theta ^{i_{1}}\wedge
...\wedge \theta ^{i_{r_{1}}}\wedge J\theta ^{j_{1}}\wedge ...\wedge J\theta
^{j_{r_{2}}}$, the term $\omega \lrcorner \omega \wedge $ will contribute $%
\tau _{1}$, the number of $a$ such that $E_{a}\lrcorner \alpha =0$ and $%
JE_{a}\lrcorner \alpha =0$, and the second term $\omega \wedge \omega
\lrcorner $ will contribute $-\tau _{2}$, the number of $b$ such that $%
JE_{b}\lrcorner E_{b}\lrcorner \alpha \neq 0$. All other contributions
cancel between the two terms. Then by counting we see that $%
n=r_{1}+r_{2}+\tau _{1}-\tau _{2}=r+\left( \tau _{1}-\tau _{2}\right) $.
Equation (1) follows.

On the other hand, since $\Lambda \alpha \in \Omega _{B}^{r-2}\left( 
\mathcal{F}\right) $ for $\alpha \in \Omega _{B}^{r}\left( \mathcal{F}%
\right) $, we have by (1)%
\begin{eqnarray*}
\left[ \left[ \Lambda ,L\right] ,\Lambda \right] \alpha &=&\left[ \Lambda ,L%
\right] \Lambda \alpha -\Lambda \left[ \Lambda ,L\right] \alpha \\
&=&\left( n-r+2\right) \Lambda \alpha -\left( n-r\right) \Lambda \alpha
=2\Lambda \alpha ,
\end{eqnarray*}%
proving (2). Taking adjoints, we obtain (3).
\end{proof}

Letting $X=\left( 
\begin{array}{cc}
0 & 1 \\ 
0 & 0%
\end{array}%
\right) $, $Y=\left( 
\begin{array}{cc}
0 & 0 \\ 
1 & 0%
\end{array}%
\right) $, $A=\left( 
\begin{array}{cc}
1 & 0 \\ 
0 & -1%
\end{array}%
\right) $ be the generators of $\mathfrak{sl}_{2}\left( \mathbb{C}\right) $,
we note that the relations are 
\begin{equation*}
\left[ X,Y\right] =A,~\left[ A,X\right] =2X,~\left[ A,Y\right] =-2Y.
\end{equation*}

\begin{lemma}
The maps $X\mapsto L$, $Y\mapsto \Lambda $, $A\mapsto \sum \left( n-r\right)
P_{r}$ induces an $\mathfrak{sl}_{2}\left( \mathbb{C}\right) $
representation on the fibers of $\Omega _{B}^{\ast }\left( \mathcal{F}%
\right) $.
\end{lemma}

\begin{proof}
The relations are easily checked using the Lemma above.
\end{proof}

In what follows, we call an element $\xi \in \Lambda ^{\ast }Q^{\ast }$ 
\textbf{primitive} if 
\begin{equation*}
\Lambda \xi =0.
\end{equation*}

\begin{corollary}
Each fiber of the bundle $\Lambda ^{\ast }Q^{\ast }$ decompose into
irreducible representations of $\mathfrak{sl}_{2}\left( \mathbb{C}\right) $, 
$\Lambda ^{\ast }Q^{\ast }=\bigoplus_{0\leq k\leq n}V_{k}$, where each $%
V_{k} $ of dimension $k+1$ has the form 
\begin{equation*}
V_{k}=\mathbb{C}\alpha +\mathbb{C}L\alpha +...+\mathbb{C}L^{k}\alpha ,
\end{equation*}%
where $\alpha \in \left( \ker \Lambda \right) \cap \Lambda ^{n-k}Q^{\ast }$
is primitive, $L^{r}\alpha \in \left( Q^{\ast }\right) ^{n-k+2r}$ for $0\leq
r\leq k$.
\end{corollary}

\begin{proof}
Direct application of the $\mathfrak{sl}_{2}\left( \mathbb{C}\right) $
representation theory.
\end{proof}

By the K\"{a}hler conditions that $\nabla J=0$ and $d\omega =0$, the tensor
field $\Lambda $ is parallel and has constant rank on $\Omega _{B}^{r}\left( 
\mathcal{F}\right) $. Hence its kernel $\ker \Lambda \subseteq \Omega
_{B}^{r}\left( \mathcal{F}\right) $ is a parallel subbundle of $\Omega
_{B}^{r}\left( \mathcal{F}\right) $. We let 
\begin{equation*}
\Omega _{B,P}^{r}\left( \mathcal{F}\right) =\Gamma _{B}\left( \ker \Lambda
\cap \Lambda ^{r}Q^{\ast }\right) \subseteq \Omega _{B}^{r}\left( \mathcal{F}%
\right) .
\end{equation*}%
denote the space of primitive basic forms.

\begin{proposition}
\label{LFormProp}Let $(M,\mathcal{F},J,g_{Q})$ be a transverse K\"{a}hler
foliation of codimension $2n$ on a compact Riemannian manifold with
bundle-like metric $g_{M}$. We have the following.

\begin{enumerate}
\item $\Omega _{B,P}^{r}\left( \mathcal{F}\right) =0$ if $r>n$.

\item If $\alpha \in \Omega _{B,P}^{r}\left( \mathcal{F}\right) $, then $%
L^{j}\alpha \neq 0$ for $0\leq j\leq n-r$ and $L^{k}\alpha =0$ for $k>n-r$.

\item The map $L^{k}:\Omega _{B}^{r}\left( \mathcal{F}\right) \rightarrow
\Omega _{B}^{r+2k}\left( \mathcal{F}\right) $ is injective for $0\leq k\leq
n-r$.

\item The map $L^{k}:\Omega _{B}^{r}\left( \mathcal{F}\right) \rightarrow
\Omega _{B}^{r+2k}\left( \mathcal{F}\right) $ is surjective for $k=n-r$.

\item $\Omega _{B}^{r}\left( \mathcal{F}\right) =\bigoplus_{k\geq
0}L^{k}\Omega _{B,P}^{r-2k}\left( \mathcal{F}\right) $.
\end{enumerate}
\end{proposition}

\begin{proof}
We apply the Lemma and Corollary above to get (1)\ and (2)\ immediately.
Statement (3) follows from (2). For (4), note that pointwise the transverse
Hodge star $\overline{\ast }$ is an isomorphism from $\Omega _{B}^{r}\left( 
\mathcal{F}\right) \rightarrow \Omega _{B}^{2n-r}\left( \mathcal{F}\right) $%
, so the bundles have the same rank. Thus $L^{n-r}$ is a vector bundle
isomorphism from $\Omega _{B}^{r}\left( \mathcal{F}\right) \rightarrow
\Omega _{B}^{2n-r}\left( \mathcal{F}\right) $, so for all $\beta \in \Omega
_{B}^{2n-r}\left( \mathcal{F}\right) $, $\left( L^{n-r}\right) ^{-1}\beta
\in \Omega _{B}^{r}\left( \mathcal{F}\right) $ gets mapped to $\beta $, so $%
L^{k}$ is surjective for $k=n-r$.

Statement (5) follows from the fact that every $\mathfrak{sl}_{2}\left( 
\mathbb{C}\right) $ representation is a direct sum of irreducible
representations.
\end{proof}

\begin{lemma}
\label{LaplacianKahlerFormLemma}Let $(M,\mathcal{F},J,g_{Q})$ be a
transverse K\"{a}hler foliation with compact leaf closures and a compatible
bundle-like metric. Then 
\begin{equation*}
\Delta _{B}=\square _{B}+\overline{\square }_{B}+\partial
_{B}H^{0,1}\lrcorner +H^{0,1}\lrcorner \,\partial _{B}+\bar{\partial}%
_{B}H^{1,0}\lrcorner +H^{1,0}\lrcorner \,\bar{\partial}_{B}.
\end{equation*}
\end{lemma}

\begin{proof}
Since $d_{B}^{2}=d_{T}^{2}=0$, $\partial _{B}^{2}=\bar{\partial}%
_{B}^{2}=\partial _{B}\bar{\partial}_{B}+\bar{\partial}_{B}\partial _{B}=0$
and $\partial _{T}^{2}=\bar{\partial}_{T}^{2}=\partial _{T}\bar{\partial}%
_{T}+\bar{\partial}_{T}\partial _{T}=0$. By direct calculation, we have 
\begin{eqnarray*}
\Delta _{B} &=&\square _{B}+\overline{\square }_{B}+(\bar{\partial}%
_{B}\partial _{B}^{\ast }+\partial _{B}^{\ast }\bar{\partial}_{B})+(\partial
_{B}\bar{\partial}_{B}^{\ast }+\bar{\partial}_{B}^{\ast }\partial _{B}) \\
&=&\square _{B}+\overline{\square }_{B}+(\bar{\partial}_{B}\left( \partial
_{T}^{\ast }+H^{1,0}\lrcorner \,\right) +\left( \partial _{T}^{\ast
}+H^{1,0}\lrcorner \,\right) \bar{\partial}_{B}) \\
&&+(\partial _{B}\left( \bar{\partial}_{T}^{\ast }+H^{0,1}\lrcorner
\,\right) +\left( \bar{\partial}_{T}^{\ast }+H^{0,1}\lrcorner \,\right)
\partial _{B}) \\
&=&\square _{B}+\overline{\square }_{B}+\bar{\partial}_{T}^{\ast }\partial
_{B}+\partial _{B}\bar{\partial}_{T}^{\ast }+\partial _{T}^{\ast }\bar{%
\partial}_{B}+\bar{\partial}_{B}\partial _{T}^{\ast } \\
&&+\bar{\partial}_{B}H^{1,0}\lrcorner \,+H^{1,0}\lrcorner \bar{\partial}%
_{B}+\partial _{B}H^{0,1}\lrcorner +H^{0,1}\lrcorner \partial _{B}
\end{eqnarray*}

Observe that the term $\bar{\partial}_{T}^{\ast }\partial _{B}+\partial _{B}%
\bar{\partial}_{T}^{\ast }$ is just the term $\bar{\partial}^{\ast }\partial
+\partial \bar{\partial}^{\ast }$ on the local quotient manifolds of the
foliation charts, and also $\partial _{T}^{\ast }\bar{\partial}_{B}+\bar{%
\partial}_{B}\partial _{T}^{\ast }$ is $\partial ^{\ast }\bar{\partial}+\bar{%
\partial}\partial ^{\ast }$. The sum of these two terms is the same as $%
\Delta -\overline{\square }-\square $ on the foliation chart quotients,
which is zero since $(M,\mathcal{F},J,g_{Q})$ is transversely K\"{a}hler.
\end{proof}

Let $\mathcal{H}_{B}^{j}\left( \mathcal{F}\right) =\ker \left( \Delta
_{B}^{j}\right) :=\ker \left( \left. \Delta _{B}\right\vert _{\Omega
_{B}^{j}\left( \mathcal{F}\right) }\right) $.

\begin{corollary}
\label{weakInequalityCorollary}Let $(M,\mathcal{F},J,g_{Q})$ be a transverse
K\"{a}hler foliation of codimension $2n$ on a compact manifold. Then for $%
0\leq j\leq 2n$, $r,s\geq 0$ such that $r+s=j$,%
\begin{equation*}
\dim \left( \mathcal{H}_{B}^{j}\left( \mathcal{F}\right) \cap \Omega
_{B}^{r,s}\left( \mathcal{F}\right) \right) \leq \dim H_{B}^{r,s}\left( 
\mathcal{F}\right) .
\end{equation*}
\end{corollary}

\begin{proof}
From the Lemma above, for any $\left( r,s\right) $-form $\phi $, we have $%
\Delta _{B}\phi \cap \Omega _{B}^{r,s}=\left( \square _{B}\phi +\overline{%
\square }_{B}\phi \right) $, so 
\begin{equation*}
\ll \Delta _{B}\phi ,\phi \gg =\ll \square _{B}\phi ,\phi \gg +\ll \overline{%
\square }_{B}\phi ,\phi \gg .
\end{equation*}%
Since $\Delta _{B}$, $\square _{B}$, $\overline{\square }_{B}$ are
nonnegative operators, if $\phi \in \mathcal{H}_{B}^{j}\left( \mathcal{F}%
\right) \cap \Omega _{B}^{r,s}\left( \mathcal{F}\right) $, then $\overline{%
\square }_{B}\phi =0$. The result follows from the Hodge theorem.
\end{proof}

\begin{theorem}
\label{HardLefschetzTheorem}(\textbf{Hard Lefschetz Theorem. }Proved in \cite%
[Th\'{e}or\'{e}me 3.4.6]{EK} for the case of minimal transverse K\"{a}hler
foliations) Let $(M,\mathcal{F},J,g_{Q})$ be a transverse K\"{a}hler
foliation of codimension $2n$ on a compact Riemannian manifold with
bundle-like metric $g_{M}$. Suppose that the class $\left[ \partial
_{B}\kappa _{B}^{0,1}\right] \in H_{\partial _{B}\bar{\partial}%
_{B}}^{1,1}\left( \mathcal{F}\right) $ is trivial. Then the Hard Lefschetz
Theorem holds for basic Dolbeault cohomology. That is, the map%
\begin{equation*}
L^{k}:H_{B}^{r}\left( \mathcal{F}\right) \rightarrow H_{B}^{r+2k}\left( 
\mathcal{F}\right)
\end{equation*}%
is injective for $0\leq k\leq n-r$ and surjective for $k\geq n-r$, $k\geq 0$%
. Furthermore,%
\begin{eqnarray}
H_{B}^{r}\left( \mathcal{F}\right) &=&\bigoplus_{k\geq
0}L^{k}H_{B,P}^{r-2k}\left( \mathcal{F}\right) ,  \label{HL1} \\
H_{B}^{r,s}\left( \mathcal{F}\right) &=&\bigoplus_{k\geq
0}L^{k}H_{B,P}^{r-k,s-k}\left( \mathcal{F}\right) .  \label{HL2}
\end{eqnarray}
\end{theorem}

\begin{proof}
If $\left[ \partial _{B}\kappa _{B}^{0,1}\right] \in H_{\partial _{B}\bar{%
\partial}_{B}}^{1,1}\left( \mathcal{F}\right) $ is trivial, we first modify
the leafwise metric as in Remark \ref{realizationOfClassRep} without
changing the transverse K\"{a}hler structure, so that $\partial _{B}\kappa
_{B}^{0,1}=0$. By Proposition \ref{LaplCommutesWithLProp}, in the new
metric, $\left[ L,\overline{\square }_{B}+\square _{B}\right] =0$, so that
by Lemma \ref{LaplacianKahlerFormLemma}, we have%
\begin{eqnarray*}
\left[ L,\Delta _{B}\right] &=&\left[ L,\partial _{B}H^{0,1}\lrcorner \right]
+\left[ L,H^{0,1}\lrcorner \,\partial _{B}\right] +\left[ L,\bar{\partial}%
_{B}H^{1,0}\lrcorner \right] +\left[ L,H^{1,0}\lrcorner \,\bar{\partial}_{B}%
\right] \\
&=&\partial _{B}\left[ L,H^{0,1}\lrcorner \right] +\left[ L,H^{0,1}\lrcorner
\,\right] \partial _{B}+\bar{\partial}_{B}\left[ L,H^{1,0}\lrcorner \right] +%
\left[ L,H^{1,0}\lrcorner \,\right] \bar{\partial}_{B} \\
&=&-i\left\{ \partial _{B}\epsilon \left( \kappa _{B}^{1,0}\right) +\epsilon
\left( \kappa _{B}^{1,0}\right) \partial _{B}\right\} +i\left\{ \overline{%
\partial }_{B}\epsilon \left( \kappa _{B}^{0,1}\right) +\epsilon \left(
\kappa _{B}^{0,1}\right) \overline{\partial }_{B}\right\} , \\
&=&-i\epsilon \left( \partial _{B}\kappa _{B}^{1,0}\right) +i\epsilon \left( 
\overline{\partial }_{B}\kappa _{B}^{0,1}\right) =0,
\end{eqnarray*}%
using (\ref{LCommutesWithDiffs}), (\ref{LCommutatorFormulas}), and the fact
that 
\begin{equation*}
\left( JH^{0,1}\right) ^{\flat }=-i\kappa _{B}^{1,0},~\left( JH^{1,0}\right)
^{\flat }=i\kappa _{B}^{0,1},
\end{equation*}%
which follows from (\ref{HKappaFormulas}). The first statement, (\ref{HL1}),
and (\ref{HL2}) follow from this calculation, the fact that $\left[ L,%
\overline{\square }_{B}\right] =\left[ L,\square _{B}\right] =0$, and
Proposition \ref{LFormProp}.
\end{proof}

\begin{corollary}
\label{classTrivialImpliesAlvClassTrivialCor}Let $(M,\mathcal{F},J,g_{Q})$
be a transverse K\"{a}hler foliation of codimension $2n$ on a compact
Riemannian manifold with bundle-like metric $g_{M}$. Then the following are
equivalent:

\begin{enumerate}
\item The class $\left[ \kappa _{B}\right] \in H_{B}^{1}\left( \mathcal{F}%
\right) $ is trivial; that is, $\left( M,\mathcal{F}\right) $ is taut.

\item The class $\left[ \partial _{B}\kappa _{B}^{0,1}\right] \in
H_{\partial _{B}\bar{\partial}_{B}}^{1,1}\left( \mathcal{F}\right) $ is
trivial.

\item The Hard Lefschetz Theorem holds for basic Dolbeault cohomology.
\end{enumerate}

\begin{proof}
(1) clearly implies (2), and (2) implies (3) by Theorem \ref%
{HardLefschetzTheorem}. Suppose that (3) holds. Then $L^{n}:H_{B}^{0}\left( 
\mathcal{F}\right) $ $\rightarrow H_{B}^{2n}\left( \mathcal{F}\right) $ is
an isomorphism. Suppose $\left( M,\mathcal{F}\right) $ is not taut; then $%
H_{B}^{2n}\left( \mathcal{F}\right) =\left\{ 0\right\} $, by \cite[Corollary
6.2]{Al}. Then in particular, $L^{n}\left( 1\right) =\omega ^{n}$ is an
exact form, as well as being a basic harmonic form, by the proof of Theorem %
\ref{HardLefschetzTheorem}. Thus, by the basic Hodge theorem, $\omega ^{n}=0$%
, a contradiction to the nondegeneracy of $\omega $. Thus, $\left( M,%
\mathcal{F}\right) $ must in fact be taut, so (1) holds.
\end{proof}
\end{corollary}

\begin{remark}
On nonK\"{a}hler transverse Hermitian foliations, it is quite possible for $%
\left( M,\mathcal{F}\right) $ to be nontaut and for $\left[ \partial
_{B}\kappa _{B}^{0,1}\right] \in H_{\partial _{B}\bar{\partial}%
_{B}}^{1,1}\left( \mathcal{F}\right) $ to be trivial, even zero. See the
Examples section.
\end{remark}

\begin{remark}
The corollary implies that tautness for transverse K\"{a}hler foliations is
characterized by the weaker condition that $\left[ \partial _{B}\kappa
_{B}^{0,1}\right] \in H_{\partial _{B}\bar{\partial}_{B}}^{1,1}\left( 
\mathcal{F}\right) $ is trivial. Also, it tell us that if the class is
nontrivial for a nontaut transverse Hermitian foliation, that foliation does
not admit a transverse K\"{a}hler structure. Thus, the Hard Lefschetz
Theorem in \cite{EK} cannot be generalized to nontaut transverse K\"{a}hler
foliations.
\end{remark}

\begin{remark}
Since $\partial _{B}\kappa _{B}^{0,1}$ is $\partial _{B}$-exact and $%
\overline{\partial }_{B}$-closed and $d$-closed, the class $\left[ \partial
_{B}\kappa _{B}^{0,1}\right] $ measures the failure of the classical $%
\partial \overline{\partial }$-lemma (or $dd_{c}$-Lemma) to hold in the case
of transverse K\"{a}hler foliations, specificially applied to the mean
curvature form. Thus, in general we do not expect the basic cohomology to be
formal or to satisfy the typical properties of that of ordinary K\"{a}hler
manifolds. In Section \ref{ddcSection}, we find sufficient conditions for
the transverse $dd_{c}$-Lemma to hold.
\end{remark}

\section{Case of automorphic mean curvature\label{automorphMeanCurvSection}}

The set of foliate vector fields is 
\begin{equation*}
V\left( \mathcal{F}\right) =\left\{ Y\in \Gamma \left( TM\right) :\left[ X,Y%
\right] \in \Gamma \left( T\mathcal{F}\right) \text{ for all }X\in \Gamma
\left( T\mathcal{F}\right) \right\} ,
\end{equation*}%
and it consists of the set of vector fields whose flows preserve $\mathcal{F}
$. For any $X\in V\left( \mathcal{F}\right) $, $\pi \left( X\right) $ is a
basic section of $Q$, meaning that $\nabla _{v}\pi \left( X\right) =0$ for
every $v$ in $T\mathcal{F}$. We say that a vector field $Y\in V\left( 
\mathcal{F}\right) $ is \textbf{transversely automorphic} if $\mathcal{L}%
_{Y}J=0$, so that $\left[ Y,J\pi \left( X\right) \right] =J\pi \left[ Y,X%
\right] $ for all $X\in V\left( \mathcal{F}\right) $. Such vector fields are
infinitesimal automorphisms of the foliation that preserve the transverse
complex structure. Sometimes we also refer to the image $\pi \left( Y\right)
\in \Gamma \left( Q\right) $ as being transversely automorphic, because the
property only depends on the properties of $\pi \left( Y\right) $.

For a complex basic normal vector field $Z\in \Gamma _{B}Q^{1,0}$, we say $Z$
is \textbf{transversely holomorphic} if $\nabla _{\bar{V}}Z=0$ for $\bar{V}%
\in Q^{0,1}$. This is equivalent to $Z$ being a basic vector field that can
be expressed as a holomorphic vector field in the transverse variables of
the local foliation charts. The following results have been previously
proved.

\begin{lemma}
\cite[Proposition 3.3]{JR2}Let $\left( M,\mathcal{F},J,g_{Q}\right) $ be a
transverse K\"{a}hler foliation with compatible bundle-like metric. The
field $X\in V\left( \mathcal{F}\right) $ is transversely automorphic if and
only if $\nabla _{JY}\pi \left( X\right) =J\nabla _{Y}\pi \left( X\right) $
for all $Y\in V\left( \mathcal{F}\right) $.
\end{lemma}

\begin{lemma}
\cite[Proposition 3.3]{JR2}\label{automorphHolomorphKahlerLemma}Let $\left(
M,\mathcal{F},J,g_{Q}\right) $ be a transverse K\"{a}hler foliation with
compatible bundle-like metric. The field $X\in V\left( \mathcal{F}\right) $
is transversely automorphic if and only if the complex normal vector field $%
Z=\pi (X)-iJ\pi (X)\in \Gamma _{B}Q^{1,0}$ is transversely holomorphic. A
complex basic normal vector field $W\in \Gamma _{B}Q^{1,0}$ is transversely
holomorphic if and only if the field $W+\overline{W}\in \Gamma _{B}Q$ is
transversely automorphic.
\end{lemma}

As in the last section, we use local orthonormal basic frames $\left\{
V_{a}\right\} $ for $Q^{1,0}$ and $\left\{ \omega _{a}\right\} $ for $\left(
Q^{1,0}\right) ^{\ast }=Q_{1,0}$.

\begin{proposition}
\cite[Lemma 3.15]{JR2}\label{transvHolomIFFProp}Let $(M,\mathcal{F},J,g_{Q})$
be a transverse Hermitian foliation. Let $Z\in \Gamma _{B}Q^{1,0}$. Then the
following are equivalent.

\begin{enumerate}
\item $Z$ is transversely holomorphic.

\item $Z\ $satisfies $\bar{\partial}_{B}Z\lrcorner +Z\lrcorner \,\bar{%
\partial}_{B}=0$.
\end{enumerate}
\end{proposition}

\begin{remark}
It is interesting to determine the relationship between the condition above
and the condition $\bar{\partial}_{B}Z^{\flat }=0$. Note that if $Z\in
\Gamma _{B}Q^{1,0}$, $Z^{\flat }\in \Omega _{B}^{0,1}\left( \mathcal{F}%
\right) $, and if $\bar{\partial}_{B}Z^{\flat }=0$, then making the usual
choices of frame (with covariant derivatives vanishing at the point in
question) we have 
\begin{eqnarray*}
0 &=&\sum_{a,b}\bar{\omega}^{a}\wedge \nabla _{\bar{V}_{a}}\left(
\left\langle Z,V_{b}\right\rangle \bar{\omega}^{b}\right) \\
&=&\sum_{a,b}\left( \bar{V}_{a}\left\langle Z,V_{b}\right\rangle \right) 
\bar{\omega}^{a}\wedge \bar{\omega}^{b} \\
&=&\sum_{a,b}\left\langle \nabla _{\bar{V}_{a}}Z,V_{b}\right\rangle \bar{%
\omega}^{a}\wedge \bar{\omega}^{b},
\end{eqnarray*}%
so that $\bar{\partial}_{B}Z^{\flat }=0$ is equivalent to $\left\langle
\nabla _{\bar{V}_{a}}Z,V_{b}\right\rangle =\left\langle \nabla _{\bar{V}%
_{b}}Z,V_{a}\right\rangle $ for all $a,b$. So it is definitely the case that
if $Z$ is transversely holomorphic, then $\bar{\partial}_{B}Z^{\flat }=0$,
but the converse is false in general. (In the Examples section, we will see
cases where $H^{1,0}$ is not transversely holomorphic but where (as always) $%
\overline{\partial }_{B}\kappa ^{0,1}=\overline{\partial }_{B}\left(
H^{1,0}\right) ^{\flat }=0$.)
\end{remark}

We would now like to apply these results about automorphic vector fields to
the mean curvature.

\begin{theorem}
\label{automorphicMeanCurvTheorem}Let $\left( M,\mathcal{F},J,g_{Q}\right) $
be a transverse K\"{a}hler foliation with compact leaf closures and a
compatible bundle-like metric. Then the mean curvature of $\left( M,\mathcal{%
F}\right) $ is automorphic if and only if%
\begin{equation*}
\Delta _{B}=\square _{B}+\overline{\square }_{B}.
\end{equation*}
\end{theorem}

\begin{proof}
Suppose the mean curvature is automorphic. Then by Lemma \ref%
{automorphHolomorphKahlerLemma} and Proposition \ref{transvHolomIFFProp},%
\begin{equation*}
\bar{\partial}_{B}H^{1,0}\lrcorner +H^{1,0}\lrcorner \,\bar{\partial}_{B}=0,
\end{equation*}%
and also by conjugating, $\partial _{B}H^{0,1}\lrcorner +H^{0,1}\lrcorner
\,\partial _{B}=0$. By the formula in Lemma \ref{LaplacianKahlerFormLemma}, $%
\Delta _{B}=\square _{B}+\overline{\square }_{B}$.\newline
Conversely, suppose that $\bar{\partial}_{B}H^{1,0}\lrcorner
+H^{1,0}\lrcorner \,\bar{\partial}_{B}+\partial _{B}H^{0,1}\lrcorner
+H^{0,1}\lrcorner \,\partial _{B}=0$. Applying this operator to $\omega ^{b}$%
, we obtain as in the proof of Proposition \ref{transvHolomIFFProp} that 
\begin{equation*}
\sum_{a}(\nabla _{\bar{V}_{a}}H^{1,0}\lrcorner \,\omega ^{b})\bar{\omega}%
^{a}+0=0,
\end{equation*}%
so that $\nabla _{\bar{V}_{a}}H^{1,0}=0$ for all $a$. Thus, $H^{1,0}$ is
transversely holomorphic, making the mean curvature automorphic.
\end{proof}

Another consequence of Lemma \ref{LaplacianKahlerFormLemma} and Proposition %
\ref{transvHolomIFFProp} is the following.

\begin{corollary}
\label{automorphicIFFpreserveTypeOfFormCor}Let $\left( M,\mathcal{F}%
,J,g_{Q}\right) $ be a transverse K\"{a}hler foliation with compact leaf
closures and a compatible bundle-like metric. Then the mean curvature of $%
\left( M,\mathcal{F}\right) $ is automorphic if and only if $\Delta _{B}$
preserves the $\left( r,s\right) $ type of a form.
\end{corollary}

\begin{corollary}
\label{cohomIneqAutomorphicCase}Let $(M,\mathcal{F},J,g_{Q})$ be a
transverse K\"{a}hler foliation of codimension $2n$ on a compact manifold
such that the mean curvature is automorphic. Then for $0\leq j\leq 2n$,%
\begin{equation*}
\dim \left( \mathcal{H}_{B}^{j}\left( \mathcal{F}\right) \right) \leq
\sum_{r+s=j;~r,s\geq 0}\dim H_{B}^{r,s}\left( \mathcal{F}\right) .
\end{equation*}
\end{corollary}

\begin{proof}
From the theorem and corollary above, under these conditions, for any
differential $j$-form $\phi $, 
\begin{equation*}
\ll \Delta _{B}\phi ,\phi \gg =\ll \square _{B}\phi ,\phi \gg +\ll \overline{%
\square }_{B}\phi ,\phi \gg .
\end{equation*}%
Since $\Delta _{B}$, $\square _{B}$, $\overline{\square }_{B}$ are
nonnegative operators, if $\phi \in \mathcal{H}_{B}^{j}\left( \mathcal{F}%
\right) $, then $\overline{\square }_{B}\phi =0$ as well, so each $\left(
r,s\right) $ component of $\phi $ is $\overline{\square }_{B}$-harmonic. The
result follows from the Hodge theorem.
\end{proof}

\section{The $dd_{c}$ Lemma\label{ddcSection}}

We would like to use the power of the $dd_{c}$ Lemma from K\"{a}hler
geometry to use in our setting. In the foliation setting, we will need some
assumptions. Let $\left( M,\mathcal{F},J,g_{Q}\right) $ be a transverse K%
\"{a}hler foliation of codimension $2n$ on a compact manifold. First we
extend the almost complex structure $J$ by pullback to the operator%
\begin{equation*}
C:\Omega _{B}^{\ast }\left( \mathcal{F}\right) \rightarrow \Omega _{B}^{\ast
}\left( \mathcal{F}\right) .
\end{equation*}%
Note that 
\begin{equation*}
C=\sum_{0\leq a,b\leq n}i^{a-b}P_{a,b},
\end{equation*}%
where $P_{a,b}:\Omega _{B}^{\ast }\left( \mathcal{F}\right) \rightarrow
\Omega _{B}^{a,b}\left( \mathcal{F}\right) $ is the projection. Then $%
C^{\ast }=C^{-1}=\sum_{0\leq a,b\leq n}i^{b-a}P_{a,b}$.

For $0\leq k\leq 2n$, let $d_{c}:\Omega _{B}^{k}\left( \mathcal{F}\right)
\rightarrow \Omega _{B}^{k+1}\left( \mathcal{F}\right) $ be defined by 
\begin{equation*}
d_{c}=i\left( \bar{\partial}_{B}-\partial _{B}\right) =C^{\ast }dC=C^{-1}dC.
\end{equation*}%
Note that $d_{c}$ is a real operator, and its adjoint with respect to basic
forms is 
\begin{equation*}
d_{c}^{\ast }=C^{\ast }\delta _{B}C=C^{-1}\delta _{B}C.
\end{equation*}%
Note that 
\begin{equation*}
\left. dd_{c}\right\vert _{\Omega _{B}^{\ast }\left( \mathcal{F}\right)
}=2i\partial _{B}\bar{\partial}_{B}=-2i\bar{\partial}_{B}\partial
_{B}=\left. -d_{c}d\right\vert _{\Omega _{B}^{\ast }\left( \mathcal{F}%
\right) }.
\end{equation*}%
Let%
\begin{equation*}
\Delta _{d_{c}}=d_{c}d_{c}^{\ast }+d_{c}^{\ast }d_{c}=C^{-1}\Delta _{B}C.
\end{equation*}

\begin{lemma}
Let $\left( M,\mathcal{F},J,g_{Q}\right) $ be a transverse K\"{a}hler
foliation with compact leaf closures and a compatible bundle-like metric,
such that the mean curvature of $\left( M,\mathcal{F}\right) $ is
automorphic. Then%
\begin{equation*}
\Delta _{B}=C^{-1}\Delta _{B}C=\Delta _{d_{c}}.
\end{equation*}
\end{lemma}

\begin{proof}
By Corollary \ref{automorphicIFFpreserveTypeOfFormCor}, $\Delta _{B}$
preserves the type of differential form, so that $C$ is just multiplication
by a scalar on forms of type $\left( r,s\right) $. The result follows.
\end{proof}

\begin{lemma}
Let $\left( M,\mathcal{F},J,g_{Q}\right) $ be a transverse K\"{a}hler
foliation with compact leaf closures and a compatible bundle-like metric,
such that the mean curvature of $\left( M,\mathcal{F}\right) $ is
automorphic. Then $\partial _{B}\kappa ^{0,1}=0$ if and only if%
\begin{equation*}
\delta _{B}d_{c}+d_{c}\delta _{B}=\nabla _{JH},
\end{equation*}%
and this is true if and only if $M$ is taut and 
\begin{equation*}
\delta _{B}d_{c}+d_{c}\delta _{B}=0.
\end{equation*}
\end{lemma}

\begin{proof}
We have%
\begin{eqnarray*}
\delta _{B}d_{c}+d_{c}\delta _{B} &=&i\left( \partial _{T}^{\ast }+\bar{%
\partial}_{T}^{\ast }+H^{1,0}\lrcorner +H^{0,1}\lrcorner \right) \left( \bar{%
\partial}_{B}-\partial _{B}\right) +i\left( \bar{\partial}_{B}-\partial
_{B}\right) \left( \partial _{T}^{\ast }+\bar{\partial}_{T}^{\ast
}+H^{1,0}\lrcorner +H^{0,1}\lrcorner \right) \\
&=&i\left( \partial _{T}^{\ast }+\bar{\partial}_{T}^{\ast }\right) \left( 
\bar{\partial}_{B}-\partial _{B}\right) +i\left( \bar{\partial}_{B}-\partial
_{B}\right) \left( \partial _{T}^{\ast }+\bar{\partial}_{T}^{\ast }\right) \\
&&+i\left( H^{1,0}\lrcorner +H^{0,1}\lrcorner \right) \left( \bar{\partial}%
_{B}-\partial _{B}\right) +i\left( \bar{\partial}_{B}-\partial _{B}\right)
\left( H^{1,0}\lrcorner +H^{0,1}\lrcorner \right) \\
&=&i\left( \overline{\square }_{T}-\square _{T}+\partial _{T}^{\ast }\bar{%
\partial}_{B}+\bar{\partial}_{B}\partial _{T}^{\ast }-\bar{\partial}%
_{T}^{\ast }\partial _{B}-\partial _{B}\bar{\partial}_{T}^{\ast }\right) \\
&&+i\left( H^{1,0}\lrcorner \bar{\partial}_{B}+\bar{\partial}%
_{B}H^{1,0}\lrcorner +H^{0,1}\lrcorner \bar{\partial}_{B}+\bar{\partial}%
_{B}H^{0,1}\lrcorner -H^{1,0}\lrcorner \partial _{B}-\partial
_{B}H^{1,0}\lrcorner -H^{0,1}\lrcorner \partial _{B}-\partial
_{B}H^{0,1}\lrcorner \right) .
\end{eqnarray*}%
By Corollary \ref{LCor}, $-i\bar{\partial}_{T}^{\ast }=\left[ \Lambda
,\partial _{B}\right] $, so $-i\left( \bar{\partial}_{T}^{\ast }\partial
_{B}+\partial _{B}\bar{\partial}_{T}^{\ast }\right) =\partial _{B}\Lambda
\partial _{B}-\partial _{B}\Lambda \partial _{B}=0$, and similarly $\partial
_{T}^{\ast }\bar{\partial}_{B}+\bar{\partial}_{B}\partial _{T}^{\ast }=0$.
Since the foliation is transversely K\"{a}hler, $\overline{\square }%
_{T}-\square _{T}=0$. By Lemma \ref{automorphHolomorphKahlerLemma} and
Proposition \ref{transvHolomIFFProp},%
\begin{eqnarray*}
\bar{\partial}_{B}H^{1,0}\lrcorner +H^{1,0}\lrcorner \,\bar{\partial}_{B}
&=&0, \\
\partial _{B}H^{0,1}\lrcorner +H^{0,1}\lrcorner \,\partial _{B} &=&0,
\end{eqnarray*}%
so that 
\begin{eqnarray*}
\delta _{B}d_{c}+d_{c}\delta _{B} &=&H^{0,1}\lrcorner \bar{\partial}_{B}+%
\bar{\partial}_{B}H^{0,1}\lrcorner -H^{1,0}\lrcorner \partial _{B}-\partial
_{B}H^{1,0}\lrcorner \\
&=&2i\func{Im}\left( H^{1,0}\lrcorner \partial _{B}+\partial
_{B}H^{1,0}\lrcorner \right) =\nabla _{JH},
\end{eqnarray*}%
by Lemma \ref{delKappa=0Lemma} and Proposition \ref%
{holomorphicKappaClosedProp}, since $\partial _{B}\left( H^{1,0}\right)
^{\flat }=\partial _{B}\kappa ^{0,1}=0$. The rest follows by Corollary \ref%
{classTrivialImpliesAlvClassTrivialCor}.
\end{proof}

Because of this Lemma, we do not expect the $dd_{c}$ Lemma from K\"{a}hler
geometry to hold in our setting, except in the special case when the mean
curvature is zero, since $\delta _{B}d_{c}+d_{c}\delta _{B}=0$ is needed
strongly. For this case, we prove the $dd_{c}$ Lemma in the usual way.

\begin{lemma}
($dd_{c}$\textbf{\ Lemma}) \label{ddcLemma}Let $\left( M,\mathcal{F}%
,J,g_{Q}\right) $ be a taut, transverse K\"{a}hler foliation on a compact
manifold, with a compatible bundle-like metric. Suppose that $\alpha \in
\Omega _{B}^{k}\left( \mathcal{F}\right) $ is $d_{c}$-exact and $d$-closed.
Then there exists a form $\beta \in \Omega _{B}^{k-2}\left( \mathcal{F}%
\right) $ with $\alpha =dd_{c}\beta $.

\begin{proof}
If $\alpha =d_{c}\gamma $, we write $\gamma =d\tau +\eta +\delta _{B}\xi $
by the Hodge decomposition, with $\eta $ basic harmonic. By hypothesis, $%
\eta $ is $\partial _{B}$ and $\overline{\partial }_{B}$-closed and thus $%
d_{c}$-closed as well, so 
\begin{eqnarray*}
d_{c}\gamma &=&d_{c}\left( d\tau +\eta +\delta _{B}\xi \right) \\
&=&d_{c}d\tau +d_{c}\delta _{B}\xi \\
&=&dd_{c}\left( -\tau \right) +d_{c}\delta _{B}\xi .
\end{eqnarray*}%
From the given and formulas and Lemma above, 
\begin{eqnarray*}
0 &=&dd_{c}\gamma =dd_{c}\delta _{B}\xi \\
&=&-d\delta _{B}d_{c}\xi ,
\end{eqnarray*}%
so 
\begin{equation*}
0=\left\langle d\delta _{B}d_{c}\xi ,d_{c}\xi \right\rangle =\left\Vert
\delta _{B}d_{c}\xi \right\Vert ^{2}=\left\Vert d_{c}\delta _{B}\xi
\right\Vert ^{2},
\end{equation*}%
so the equation above is 
\begin{equation*}
\alpha =d_{c}\gamma =dd_{c}\left( -\tau \right) .
\end{equation*}
\end{proof}
\end{lemma}

From this it follows that the basic cohomology is formal, as in the case of
ordinary cohomology of K\"{a}hler manifolds.

\section{The Hodge diamond\label{HodgeDiamondSection}}

On a transverse K\"{a}hler foliation with compact leaf closures and
compatible bundle-like metric, by Corollary \ref%
{automorphicIFFpreserveTypeOfFormCor} the basic Laplacian $\Delta _{B}$
preserves the $\left( r,s\right) $ type of form if and only if the mean
curvature is automorphic. For the purposes of what follows, we will consider
the case when the basic mean curvature is automorphic, and we consider the $%
\Delta _{B}$-harmonic forms of type $\left( r,s\right) $. Let%
\begin{equation*}
\mathcal{H}_{\Delta _{B}}^{r,s}\left( \mathcal{F}\right) =\left\{ \alpha \in
\Omega _{B}^{r,s}\left( \mathcal{F}\right) :\Delta _{B}\alpha =0\right\} .
\end{equation*}

\begin{theorem}
\label{HodgeDiamondTheorem}(\textbf{Hodge Diamond Theorem}) Let $\left( M,%
\mathcal{F},J,g_{Q}\right) $ be a transverse K\"{a}hler foliation of
codimension $2n$ on a compact manifold. If there exists a compatible
bundle-like metric such that the mean curvature of $\left( M,\mathcal{F}%
\right) $ is automorphic, then the spaces $\mathcal{H}_{\Delta
_{B}}^{r,s}\left( \mathcal{F}\right) $ and basic cohomology groups have the
following structure:

\begin{enumerate}
\item (\textbf{Hodge symmetry}) For all $r,s$ such that $0\leq r\leq s\leq n$%
, $\mathcal{H}_{\Delta _{B}}^{r,s}$ $\left( \mathcal{F}\right) \cong 
\mathcal{H}_{\Delta _{B}}^{s,r}\left( \mathcal{F}\right) $.

\item For all $j$ such that $0\leq j\leq 2n$, $\dim H_{B}^{j}\left( \mathcal{%
F}\right) =\sum\limits_{r+s=j;~r,s\geq 0}\dim \mathcal{H}_{\Delta
_{B}}^{r,s}\left( \mathcal{F}\right) $.

\item $\dim H_{B}^{r}\left( \mathcal{F}\right) $ is even if $r$ is odd, and $%
\dim H_{B}^{1}\left( \mathcal{F}\right) =2\dim \mathcal{H}_{\Delta
_{B}}^{1,0}\left( \mathcal{F}\right) $ is a topological invariant.

\item If in addition the class $\left[ \partial _{B}\kappa _{B}^{0,1}\right]
\in H_{\partial _{B}\bar{\partial}_{b}}^{1,1}\left( \mathcal{F}\right) $ is
trivial, then the \'{A}lvarez class $\left[ \kappa _{B}\right] \in
H_{B}^{1}\left( \mathcal{F}\right) $ is trivial, the spaces $\mathcal{H}%
_{\Delta _{B}}^{r,s}\cong H_{B}^{r,s}\left( \mathcal{F}\right) $ have the
following structure:

\begin{enumerate}
\item The map $L^{k}:H_{B}^{r}\left( \mathcal{F}\right) \rightarrow
H_{B}^{r+2k}\left( \mathcal{F}\right) $ is injective for $0\leq k\leq n-r$
and surjective for $k\geq n-r$, $k\geq 0$.

\item We have $H_{B}^{r}\left( \mathcal{F}\right) =\bigoplus_{k\geq
0}L^{k}H_{B,P}^{r-2k}\left( \mathcal{F}\right) $, $H_{B}^{p,q}\left( 
\mathcal{F}\right) =\bigoplus_{k\geq 0}L^{k}H_{B,P}^{p-k,q-k}\left( \mathcal{%
F}\right) .$

\item (\textbf{Kodaira-Serre duality}) For all $r,s$ such that $0\leq r\leq
s\leq n$, $H_{B}^{r,s}\left( \mathcal{F}\right) \cong H_{B}^{n-r,n-s}\left( 
\mathcal{F}\right) $.

\item All of the $\mathcal{H}_{\Delta _{B}}^{r,s}\left( \mathcal{F}\right) $
in (1),(2),(3) above may be replaced with $H_{B}^{r,s}\left( \mathcal{F}%
\right) $.
\end{enumerate}

\item If the class $\left[ \partial _{B}\kappa _{B}^{0,1}\right] \in
H_{\partial _{B}\bar{\partial}_{b}}^{1,1}\left( \mathcal{F}\right) $ is
nontrivial, then (4a), (4b), (4c) are false.
\end{enumerate}
\end{theorem}

\begin{proof}
By Theorem \ref{automorphicMeanCurvTheorem}, $\Delta _{B}$ is a real
operator, so that (1) holds by conjugation. By Corollary \ref%
{automorphicIFFpreserveTypeOfFormCor}, $\Delta _{B}$-harmonic forms
correspond exactly to sums of $\Delta _{B}$-harmonic forms of type $\left(
r,s\right) $, so that (2) follows, and (1) and (2) imply (3). By the proof
of the Hard Lefschetz theorem Theorem \ref{HardLefschetzTheorem}, $\left[
L,\Delta _{B}\right] =0$ when $\partial _{B}\kappa _{B}^{0,1}=0$, so we
choose the leafwise metric so that this equation holds, and this change does
not alter the dimensions of the harmonic forms (and certainly not the basic
cohomology groups). Then (4a) and (4b) follow as in the Hard Lefschetz
theorem. Corollary \ref{classTrivialImpliesAlvClassTrivialCor} implies $%
\left[ \kappa _{B}\right] \in H_{B}^{1}\left( \mathcal{F}\right) $ is
trivial, so we modify the bundle-like metric (without changing the
cohomology groups) so that $\kappa _{B}=\kappa =0$, and then Theorem \ref%
{classTrivialTheorem} and Lemma \ref{LaplacianKahlerFormLemma} imply in
addition that $\Delta _{B}=2\overline{\square }_{B}=2\square _{B}$. Then
(4c) and (4d) follow. Statement (5) is a consequence of Corollary \ref%
{classTrivialImpliesAlvClassTrivialCor}.
\end{proof}

\begin{remark}
Items (4a) through (4d) of the theorem above were essentially already known,
because the minimal foliation case was shown in \cite{EK}.
\end{remark}

\begin{remark}
The theorem above gives topological obstructions to the existence of
transverse K\"{a}hler foliations with automorphic mean curvature, and
further obstructions if we require that $\left[ \partial _{B}\kappa
_{B}^{0,1}\right] \in H_{\partial _{B}\bar{\partial}_{b}}^{1,1}\left( 
\mathcal{F}\right) $ is trivial.
\end{remark}

\section{Examples\label{examplesSection}}

\begin{example}
\label{KaehlerExactMeanCurv} Note that in contrast to the situation of a K%
\"{a}hler form on an ordinary manifold, it is possible that $\omega $ is a
trivial class in basic cohomology. This always happens when we consider
nontaut codimension 2 foliations. We consider the Carri\`{e}re example from 
\cite{Ca}. Let $A$ be a matrix in $\mathrm{SL}_{2}(\mathbb{Z})$ of trace
strictly greater than $2$. We denote respectively by $v_{1}$ and $v_{2}$
unit eigenvectors associated with the eigenvalues $\lambda $ and $\frac{1}{%
\lambda }$ of $A$ with $\lambda >1$ irrational. Let the hyperbolic torus $%
\mathbb{T}_{A}^{3}$ be the quotient of $\mathbb{T}^{2}\times \mathbb{R}$ by
the equivalence relation which identifies $(m,t)$ to $(A(m),t+1)$. The flow
generated by the vector field $V_{2}$ is a Riemannian foliation with
bundle-like metric (letting $\left( x,s,t\right) $ denote local coordinates
in the $v_{2}$ direction, $v_{1}$ direction, and $\mathbb{R}$ direction,
respectively) 
\begin{equation*}
g=\lambda ^{-2t}dx^{2}+\lambda ^{2t}ds^{2}+dt^{2}.
\end{equation*}%
Note that the mean curvature of the flow is $\kappa =\kappa _{B}=\log \left(
\lambda \right) dt$, since $\chi _{\mathcal{F}}=\lambda ^{-t}dx$ is the
characteristic form and $d\chi _{\mathcal{F}}=-\log \left( \lambda \right)
\lambda ^{-t}dt\wedge dx=-\kappa \wedge \chi _{\mathcal{F}}$. Then an
orthonormal frame field for this manifold is $\{X=\lambda ^{t}\partial
_{x},S=\lambda ^{-t}\partial _{s},T=\partial _{t}\}$ corresponding to the
orthonormal coframe $\{X^{\ast }=\chi _{\mathcal{F}}=\lambda ^{-t}dx,S^{\ast
}=\lambda ^{t}ds,T^{\ast }=dt\}$. Then, letting $J$ be defined by $%
J(S)=T,J(T)=-S$, the Nijenhuis tensor 
\begin{equation*}
N_{J}(S,T)=[S,T]+J\left( [JS,T]+[S,JT]\right) -[JS,JT]
\end{equation*}%
clearly vanishes, so that $J$ is integrable. (This is also easy to see with
other means.)

The corresponding transverse K\"{a}hler form is seen to be $\omega =T^{\ast
}\wedge S^{\ast }=\lambda ^{t}dt\wedge ds=d(\frac{1}{\log \lambda }S^{\ast
}) $, an exact form in basic cohomology. From the above, 
\begin{eqnarray*}
\kappa &=&\kappa _{B}=\log \left( \lambda \right) dt=\log \left( \lambda
\right) T^{\ast } \\
S^{\ast } &=&\lambda ^{t}ds,~Z^{\ast }=\frac{1}{2}\left( S^{\ast }+iT^{\ast
}\right) =\frac{1}{2}\left( \lambda ^{t}ds+idt\right) ,
\end{eqnarray*}%
so%
\begin{eqnarray*}
\kappa _{B} &=&-i\left( \log \lambda \right) Z^{\ast }+i\left( \log \lambda
\right) \bar{Z}^{\ast } \\
&=&-i\left( \log \lambda \right) \frac{1}{2}\left( \lambda ^{t}ds+idt\right)
+i\left( \log \lambda \right) \bar{Z}^{\ast }.
\end{eqnarray*}%
Then%
\begin{eqnarray*}
\kappa _{B}^{1,0} &=&-i\log \left( \lambda \right) Z^{\ast }=-\frac{i}{2}%
\left( \log \lambda \right) \left( \lambda ^{t}ds+idt\right) \\
\bar{\partial}_{B}\kappa _{B}^{1,0} &=&d\kappa _{B}^{1,0}=\frac{i}{2}\left(
\log \lambda \right) ^{2}\lambda ^{t}ds\wedge dt \\
&=&\frac{i}{2}\left( \log \lambda \right) ^{2}S^{\ast }\wedge T^{\ast } \\
&=&\left( \log \lambda \right) ^{2}\bar{Z}^{\ast }\wedge Z^{\ast }.
\end{eqnarray*}%
It is impossible to change the metric so that this is zero. The reason is
that from \cite{Al} the mean curvature $\kappa _{B}^{\prime }$ for any other
compatible bundle-like metric would satisfy $\kappa _{B}^{\prime }=\kappa
_{B}+df$ for some real basic function $f$, which would imply that $(\kappa
_{B}^{1,0})^{\prime }=\kappa _{B}^{1,0}+\partial _{B}f$, and $\partial
_{B}f=Z(f)Z^{\ast }$. Since $f$ is a periodic function of $t$ alone, this is 
$\partial _{B}f=-i(\partial _{t}f)\,Z^{\ast }$. Then in that case 
\begin{eqnarray*}
\bar{\partial}_{B}(\kappa _{B}^{1,0})^{\prime } &=&d(\kappa
_{B}^{1,0}-i(\partial _{t}f)\,Z^{\ast }) \\
&=&d\left( -\frac{i}{2}\left( \log \lambda +2\partial _{t}f\right) \left(
\lambda ^{t}ds+idt\right) \right) \\
&=&\left( \frac{i}{2}(\log \lambda )^{2}+i\partial _{t}^{2}f+i(\log \lambda
)\partial _{t}f)\right) \lambda ^{t}ds\wedge dt \\
&=&\left( (\log \lambda )^{2}+2\partial _{t}^{2}f+2(\log \lambda )\partial
_{t}f\right) \bar{Z}^{\ast }\wedge Z^{\ast }
\end{eqnarray*}%
Since the term in parentheses is never zero for any periodic function $f$,
we conclude that $\bar{\partial}\kappa _{B}^{1,0}$ is a nonzero multiple of $%
\bar{Z}^{\ast }\wedge Z^{\ast }$ for any compatible bundle-like metric. This
is not surprising, because this being zero would imply $\left( M,\mathcal{F}%
\right) $ is taut by Corollary \ref{classTrivialImpliesAlvClassTrivialCor}.

For later use, we compute that basic Dolbeault cohomology in this example.
One can easily verify that 
\begin{eqnarray*}
H_{B}^{0,0}=\ker \overline{\partial }_{B}^{0,0} &=&\func{span}\{1\} \\
H_{B}^{1,0}=\ker \overline{\partial }_{B}^{1,0} &=&\{0\} \\
H_{B}^{0,1}=\frac{\ker \overline{\partial }_{B}^{0,1}}{\func{im}\overline{%
\partial }_{B}^{0,0}} &=&\func{span}\{S^{\ast }-iT^{\ast }\} \\
H_{B}^{1,1}=\frac{\Omega _{B}^{1,1}}{\func{im}\overline{\partial }_{B}^{1,0}}
&=&\{0\},
\end{eqnarray*}%
where the last equality is true because one can show that every element of $%
\Omega _{B}^{1,1}$ is $\overline{\partial }_{B}$-exact. Now we compute $%
H_{\partial _{B}\overline{\partial }_{B}}^{\ast ,\ast }\left( \mathcal{F}%
\right) $: because $\partial _{B}\overline{\partial }_{B}f=-\frac{1}{4}%
\Delta _{B}f~dz\wedge d\overline{z}$ integrates to zero, we have%
\begin{equation*}
H_{\partial _{B}\overline{\partial }_{B}}^{0,0}\left( \mathcal{F}\right)
\cong \mathbb{C;~}H_{\partial _{B}\overline{\partial }_{B}}^{1,1}\left( 
\mathcal{F}\right) =\func{span}\{\bar{Z}^{\ast }\wedge Z^{\ast }\}\cong 
\mathbb{C},
\end{equation*}%
and as expected, $\left[ \partial _{B}\kappa _{B}^{0,1}\right] $ is nonzero
and thus a generator.

Then observe that the ordinary basic cohomology Betti numbers for this
foliation are $h_{B}^{0}=h_{B}^{1}=1$, $h_{B}^{2}=0$, we see that the basic
Dolbeault Betti numbers satisfy 
\begin{equation*}
h_{B}^{0,0}=h_{B}^{0,1}=1,\quad h_{B}^{1,0}=h_{B}^{1,1}=0.
\end{equation*}%
So even though it is true that 
\begin{equation*}
h_{B}^{j}=\sum_{r+s=j}h_{B}^{r,s},
\end{equation*}%
and the foliation is transversely K\"{a}hler, we also have (with $n=1$) 
\begin{equation*}
h_{B}^{r,s}\neq h_{B}^{s,r},\quad h_{B}^{r,s}\neq h_{B}^{n-r,n-s}.
\end{equation*}%
Theorem \ref{HodgeDiamondTheorem} tells us that the mean curvature is not
automorphic. We can also verify this directly: 
\begin{eqnarray*}
H^{1,0} &=&\frac{i\log \lambda }{2}Z, \\
\left( \overline{\partial }_{B}H^{1,0}\lrcorner +H^{1,0}\lrcorner \overline{%
\partial }_{B}\right) \kappa _{B}^{1,0} &=&\overline{\partial }_{B}\left( 
\frac{\left( \log \lambda \right) ^{2}}{2}\right) +H^{1,0}\lrcorner \left(
\log \lambda \right) ^{2}\bar{Z}^{\ast }\wedge Z^{\ast } \\
&=&-\frac{i}{2}\left( \log \lambda \right) ^{3}\bar{Z}^{\ast }\neq 0.
\end{eqnarray*}%
Another way to see this is to choose local transverse holomorphic
coordinates. The reader may check that if we choose%
\begin{equation*}
x_{0}=-\left( \log \lambda \right) s;~y_{0}=\lambda ^{-t}
\end{equation*}%
and let $z_{0}=x_{0}+iy_{0}$, then 
\begin{eqnarray*}
\partial _{x_{0}} &=&-\frac{1}{\log \lambda }\partial _{s},~\partial
_{y_{0}}=-\frac{\lambda ^{t}}{\log \lambda }\partial _{t};~dx_{0}=-\left(
\log \lambda \right) ds,~dy_{0}=-\left( \log \lambda \right) \lambda ^{-t}dt;
\\
\text{ }J\left( \partial _{x_{0}}\right) &=&\partial _{y_{0}},~J\left(
\partial _{y_{0}}\right) =-\partial _{x_{0}};~dz_{0}=-\left( \log \lambda
\right) \left( ds+i\lambda ^{-t}dt\right) ,
\end{eqnarray*}%
and so 
\begin{equation*}
\kappa _{B}^{1,0}=\frac{i}{2y_{0}}dz_{0},
\end{equation*}%
which is clearly not a holomorphic one-form.

The exactness of the basic K\"{a}hler form causes the Kodaira-Serre
argument, the Lefschetz theorem, the Hodge diamond ideas to fail. Thus, for
a nontaut, transverse K\"{a}hler foliation, it is not necessarily true that
the odd basic Betti numbers are even, and the basic Dolbeault numbers do not
have the same kinds of symmetries as Dolbeault cohomology on K\"{a}hler
manifolds. Also, this example shows that the even degree basic cohomology
groups are not always nonzero, as is the case for ordinary cohomology for
symplectic manifolds (and thus all K\"{a}hler manifolds).
\end{example}

\begin{example}
\label{ExampleNewClassZeroButNotTaut}We now consider the product foliation
on the product manifold $M=T_{A}^{3}\times T_{A}^{3}$. We will put two
different transverse Hermitian structures on $M$, and the cohomological
properties of the two transverse structures are different. In both cases we
have fixed the product metric.

\begin{enumerate}
\item First, we consider the product of the two transverse holomorphic
structures on each copy of $T_{A}^{3}$ separately. A simple calculation
shows that the foliation is transversely K\"{a}hler, nontaut. The mean
curvature is not automorphic, and the class $\left[ \partial _{B}\kappa
_{B}^{0,1}\right] $ on $M$ is nontrivial. The Betti numbers are%
\begin{eqnarray*}
h_{B}^{0} &=&1,~h_{B}^{1}=2,~h_{B}^{2}=1, \\
h_{B}^{0,0} &=&1,~h_{B}^{0,1}=2,~h_{B}^{0,2}=1, \\
h_{\partial _{B}\overline{\partial }_{B}}^{0,0} &=&1,~h_{\partial _{B}%
\overline{\partial }_{B}}^{1,1}=2,~h_{\partial _{B}\overline{\partial }%
_{B}}^{1,1}=1,
\end{eqnarray*}%
with all the other Betti numbers zero.

\item Next, instead we use the following transverse complex structure. Using
the same notation as in Example \ref{KaehlerExactMeanCurv} but using
subscripts $1$ and $2$ to refer to the different copies of $T_{A}^{3}$ in
the product, we define%
\begin{equation*}
J^{\prime }\left( U_{1}\right) =U_{2};~J^{\prime }\left( U_{2}\right)
=-U_{1},
\end{equation*}%
where $U$ denotes one of the unit normal vector fields $S$ or $T$. We then
have that the form $\omega $ is%
\begin{equation*}
\omega =\lambda ^{-t_{1}-t_{2}}dx_{2}\wedge dx_{1}+\lambda
^{t_{1}+t_{2}}ds_{2}\wedge ds_{1}+dt_{2}\wedge dt_{1},
\end{equation*}%
which is clearly not closed, so the new transverse Hermitian structure is
not K\"{a}hler. The foliation is the same as before, so it is again not
taut. The mean curvature is%
\begin{eqnarray*}
\kappa _{B} &=&\left( \log \lambda \right) \left( dt_{1}+dt_{2}\right) , \\
\kappa _{B}^{1,0} &=&\frac{1}{2}\left( \kappa _{B}+iJ^{\prime }\kappa
_{B}\right) =\frac{1}{2}\left( \log \lambda \right) \left( 1-i\right) \left(
dt_{1}+idt_{2}\right) .
\end{eqnarray*}%
This vector field is clearly holomorphic, and we also have%
\begin{equation*}
\overline{\partial }_{B}\kappa _{B}^{1,0}=d\kappa _{B}^{1,0}=0,
\end{equation*}%
so that with this new holomorphic structure, the $\partial _{B}\overline{%
\partial }_{B}$-class $\left[ \partial _{B}\kappa _{B}^{0,1}\right] $ is
trivial (even though the foliation is not taut). The Betti numbers now
satisfy%
\begin{eqnarray*}
h_{B}^{0} &=&1,~h_{B}^{1}=2,~h_{B}^{2}=1, \\
h_{B}^{0,0}
&=&1,~h_{B}^{0,1}=1=h_{B}^{1,0},~h_{B}^{1,1}=2,~h_{B}^{2,0}=h_{B}^{0,2}=1, \\
h_{\partial _{B}\overline{\partial }_{B}}^{0,0} &=&1,~h_{\partial _{B}%
\overline{\partial }_{B}}^{1,1}=1,
\end{eqnarray*}%
with all other Betti numbers zero.
\end{enumerate}

This set of examples shows that it is possible for the class $\left[
\partial _{B}\kappa _{B}^{0,1}\right] $ to be trivial for some transverse
holomorphic structures and to be nontrivial in others. But if this is the
case, by Corollary \ref{classTrivialImpliesAlvClassTrivialCor} it must be
nontrivial when the structure is transversely K\"{a}hler.
\end{example}

\end{document}